\documentclass[reqno]{amsart}
\usepackage{latexsym,amsmath,amssymb,amscd, amsthm}
\usepackage[utf8]{inputenc}
\usepackage[all]{xy}
\usepackage{enumerate}

\usepackage{hyperref} 
\makeatletter
\@addtoreset{figure}{section}
\def\thefigure{\thesection.\@arabic\c@figure}
\def\fps@figure{h,t}
\@addtoreset{table}{bsection}

\def\thetable{\thesection.\@arabic\c@table}
\def\fps@table{h, t}
\@addtoreset{equation}{section}

\makeatother

\makeatletter
\newcommand\@dotsep{4.5}
\def\@tocline#1#2#3#4#5#6#7{\relax
	\ifnum #1>\c@tocdepth 
	\else
	\par \addpenalty\@secpenalty\addvspace{#2}%
	\begingroup \hyphenpenalty\@M
	\@ifempty{#4}{%
		\@tempdima\csname r@tocindent\number#1\endcsname\relax
	}{%
		\@tempdima#4\relax
	}%
	\parindent\z@ \leftskip#3\relax \advance\leftskip\@tempdima\relax
	\rightskip\@pnumwidth plus1em \parfillskip-\@pnumwidth
	#5\leavevmode\hskip-\@tempdima #6\relax
	\leaders\hbox{$\m@th
		\mkern \@dotsep mu\hbox{.}\mkern \@dotsep mu$}\hfill
	\hbox to\@pnumwidth{\@tocpagenum{#7}}\par
	\nobreak
	\endgroup
	\fi}
\makeatother 

\theoremstyle{plain}
\newtheorem{theorem}{Theorem}
\newtheorem*{theorem*}{Theorem}
\newtheorem{corollary}[theorem]{Corollary}

\newtheorem{example}[theorem]{Example}
\newtheorem{lemma}[theorem]{Lemma}

\newtheorem{proposition}[theorem]{Proposition}
\newtheorem{remark}[theorem]{Remark}

\numberwithin{theorem}{section}
\numberwithin{equation}{section}

\newcommand{\0}{{\bf 0}}
\renewcommand{\1}{{\bf 1}}

\newcommand{\Ad}{{\rm Ad}}
\newcommand{\ad}{{\rm ad}}

\newcommand{\Ci}{{\mathcal C}^\infty}

\newcommand{\Aut}{{\rm Aut}}
\newcommand{\Bun}{\text{{\boldmath{$\mathfrak{B}$}}}}
\newcommand{\Der}{{\rm Der}}

\newcommand{\Cl}{{\rm Cl}}

\newcommand{\de}{{\rm d}}
\newcommand{\ee}{{\rm e}}
\newcommand{\End}{{\rm End}}
\newcommand{\ev}{{\rm ev}}

\newcommand{\Galg}{G^{\rm alg} }
\newcommand{\ggalg}{\gg^{\rm alg}}
\newcommand{\GL}{{\rm GL}}
\newcommand{\Hom}{{\rm Hom}}
\newcommand{\I}{{\rm I}}

\renewcommand{\Im}{{\rm Im}}
\newcommand{\Ind}{{\rm Ind}}

\newcommand{\id}{{\rm id}}
\newcommand{\ie}{{\rm i}}
\newcommand{\kk}{{\mathbf k}}
\newcommand{\Ker}{{\rm Ker}\,}

\newcommand{\nor}{{\rm nor}}

\renewcommand{\O}{{\mathbf{O}}}

\newcommand{\Prim}{{\rm Prim}}
\newcommand{\pr}{{\rm pr}}
\newcommand{\Ran}{{\rm Ran}\,}
\renewcommand{\Re}{{\rm Re}}
\newcommand{\Rel}{{\boldmath{$\mathfrak{R}$}}}
\newcommand{\RelS}{\text{{\boldmath{$\mathcal{S}$}}}}
\newcommand{\rk}{{\rm rk}}
\newcommand{\spa}{{\rm span}\,}

\newcommand{\supp}{{\rm supp}\,}

\newcommand{\CC}{{\mathbb C}}

\newcommand{\NN}{{\mathbb N}}
\newcommand{\QQ}{{\mathbb Q}}
\newcommand{\RR}{{\mathbb R}}
\newcommand{\TT}{{\mathbb T}}
\newcommand{\ZZ}{{\mathbb Z}}

\newcommand{\Bc}{{\mathcal B}}
\newcommand{\Cc}{{\mathcal C}}

\newcommand{\Fc}{{\mathcal F}}

\newcommand{\Hc}{{\mathcal H}}

\newcommand{\Jc}{{\mathcal J}}

\newcommand{\Lc}{{\mathcal L}}

\newcommand{\Oc}{{\mathcal O}}
\newcommand{\Pc}{{\mathcal P}}

\newcommand{\Vc}{{\mathcal V}}

\newcommand{\ag}{{\mathfrak a}}

\newcommand{\dg}{{\mathfrak d}}
\renewcommand{\gg}{{\mathfrak g}}
\newcommand{\hg}{{\mathfrak h}}
\newcommand{\kg}{{\mathfrak k}}

\renewcommand{\ng}{{\mathfrak n}}

\newcommand{\zg}{{\mathfrak z}}

\newcommand{\matt}[2]
{\ensuremath{\begin{pmatrix}
			{#1} & 0 \\
			   0 & {#2}
		\end{pmatrix}}}

\newcommand{\mattt}[3]
{\ensuremath{\begin{pmatrix}
			{#1} & & 0\\
			& {#2} & \\
			0& & {#3}
\end{pmatrix}}}

\usepackage{mathtools}

\newcommand{\mathsout}[1]
{\bgroup\mathchoice
	{\sbox0{$\displaystyle{#1}$}%
		\usebox0\hspace{-\wd0}%
		\rule[0.5\ht0-0.5\dp0-.5pt]{\wd0}{1pt}}%
	{\sbox0{$\textstyle{#1}$}%
		\usebox0\hspace{-\wd0}%
		\rule[0.5\ht0-0.5\dp0-.5pt]{\wd0}{1pt}}%
	{\sbox0{$\scriptstyle{#1}$}%
		\usebox0\hspace{-\wd0}%
		\rule[0.5\ht0-0.5\dp0-.5pt]{\wd0}{1pt}}%
	{\sbox0{$\scriptscriptstyle{#1}$}%
		\usebox0\hspace{-\wd0}%
		\rule[0.5\ht0-0.5\dp0-.5pt]{\wd0}{1pt}}%
	\egroup}

\title[Square-integrable representations of solvable Lie groups]
{Square-integrable representations and the coadjoint action of solvable Lie groups}

\author{Ingrid Belti\c t\u a}
\author{Daniel Belti\c t\u a}
\address{Institute of Mathematics ``Simion Stoilow'' of the Romanian Academy,
P.O. Box 1-764, Bucharest, Romania}

\email{Ingrid.Beltita@imar.ro, ingrid.beltita@gmail.com}
\email{Daniel.Beltita@imar.ro, beltita@gmail.com}
\thanks{The research of the second-named author was supported by a grant of the  Ministry of Research, Innovation and Digitization, CNCS/CCCDI -- UEFISCDI, project number PN-III-P4-ID-PCE-2020-0878, within PNCDI III}

\begin{document}

\begin{abstract}
We characterize the square-integrable representations of  
(con\-nect\-ed, simply connected) solvable Lie groups in terms of
the generalized orbits of the coadjoint action.
We prove that the normal representations corresponding, via the Puk\'anszky correspondence, 
to open coadjoint orbits are type I, not necessarily square-integrable representations. 
 We show that the quasi-equivalence classes of type~$\I$
square-integrable  representations are in bijection with the simply 
connected open coadjoint orbits, and the existence of an open coadjoint orbit guarantees the existence of a
compact open subset of the space of primitive ideals of the group.
When the nilradical has codimension~1, we prove that the isolated 
points of the primitive ideal space are
always of type~$\I$. 
This is not always true for codimension greater than 2, as shown by specific examples of solvable Lie groups that have dense, but not locally closed,  coadjoint orbits. 
\end{abstract}

\maketitle


\section{Introduction}
\label{sect1}

There has been continued interest in the property of admissibility for unitary representations of Lie groups, which requires in particular square-integrable coefficients of the representation under consideration. 
Several interesting discretized versions of square integrability properties of regular or quasi-regular representations of groups were established under the guise of square summability properties via existence of frames; see for instance \cite{FO22}, \cite{FvV21}, and \cite{Bk04}.  

We recall that when $G$ is a semisimple Lie group, its discrete series representations (i.e.,  square-integrable irreducible representations) were parameterized by Harish-Chandra, and were intensely studied ever since.
Likewise, the Kirillov-Bernat map shows that square-integrable representations of exponential solvable Lie groups correspond to open coadjoint orbits. 
However, less is known about the relation between square integrable representations and the coadjoint action in the case of general solvable Lie groups, and in this paper we contribute to the picture in this area of representation theory. 

It is desirable to have a characterization of the square-integrable representations in terms of geometric objects that are intrinsically associated to the group under investigation. 
We address this problem for solvable Lie groups that are connected and simply connected, but otherwise arbitrary, and in particular are not required to be type~$\I$.

The earlier research mainly focused on Lie groups of type~$\I$, and often related integrability 
properties of representations to some kind of open orbits 
(see e.g., \cite{DuRa76} or, more recently, 
\cite{CFT16}). 
As a consequence of our results, for a connected and simply connected solvable Lie group $G$, we give a complete characterization of the square-integrable irreducible representations in terms of coadjoint orbits, namely we show that
there is a one-to-one correspondence 
between the equivalence classes of square-integrable unitary irreducible representations of $G$  and 
the simply connected open coadjoint orbits of $G$.
(See Corollary~\ref{open_cor2}.) 
We emphasize that the group $G$ itself  is not assumed to be type~$\I$. 
Even in the case of solvable Lie groups of type~$\I$  the correspondence between 
the space of coadjoint orbits and the equivalence classes of unitary irreducible 
representations is not bijective 
(see \cite{AuKo71}), let alone a homeomorphism, so there is no direct way to translate the topological properties of the coadjoint orbits into topological properties of the primitive ideal space of the group $C^*$-algebra $C^*(G)$, or into integrability properties of the group representations.

In this paper, by square-integrable representation of 
a locally compact group~$G$ we mean a (continuous unitary) factor representation $\pi\colon G\to\Bc(\Hc)$ for which there exist vectors $f,h\in\Hc$ whose corresponding coefficient $(\pi(\cdot)f\mid h)$ of the representation~$\pi$ is square integrable on~$G$ and not identically~$0$, or equivalently, $\pi$ is  quasi-equivalent with a subrepresentation of the  regular representation (\cite[Prop.~2.3]{Ros78}, and 
 \cite[Th. 3]{Mo77} for alternative characterizations).

In order to describe the contents of the present paper in more detail let us introduce some notation. 
We denote by $\stackrel{\frown}{G}$ the set of all quasi-equivalence classes $[\pi]^\frown$ of factor unitary representations $\pi$ of a locally compact group $G$, and 
we recall from above that 
the property of square integrability is invariant under quasi-equivalence of representations.  
It thus makes sense to speak about square-integrable classes  $[\pi]^\frown\in\stackrel{\frown}{G}$. 
Let  $\stackrel{\frown}{G}_\nor\subseteq \stackrel{\frown}{G}$ be the subset corresponding to normal representations, that is, unitary representations $\pi\colon G\to\Bc(\Hc_\pi)$ for which the von Neumann algebra $\pi(G)''\subseteq \Bc(\Hc_\pi)$ is a factor with a semifinite normal faithful trace $\tau$ satisfying $0<\tau(\pi(a^*a))<\infty$ for some element $a\in C^*(G)$.

Assume that $G$ is a connected and simply connected solvable Lie group with its corresponding  
set $\RelS$ of generalized orbits of the coadjoint action  
(cf. \cite{Pu71}; see Section~\ref{Prel} below). 
The remarkable result of Puk\'anszky (see \cite[Th. 3]{Pu74}) shows that there is a 
bijective map  
$\ell\colon \RelS\to \stackrel{\frown}{G}_\nor$, now called
 the \emph{Puk\'anszky correspondence};
it recovers the Kirillov correspondence in the special case of nilpotent Lie groups. 
We note that $\ell$ is not a homeomorphism in general, 
since no appropriate topology has been identified for the entire space 
$\RelS$.

One of our main results (Theorem~\ref{open_th}) answers the natural question of describing the generalized orbits of the coadjoint action that correspond to the square-integrable classes in $\stackrel{\frown}{G}_\nor$ via the Puk\'anszky correspondence. 
The first result in this connection appeared in \cite{Ros78}, cf. Lemma~\ref{open_lemma} below; however, that description is not intrinsic, 
as the characterization is in terms of orbits of a non-canonical group~$\Galg$. 
We show that all the normal representations associated to an open coadjoint orbit are type~$\I$ (Theorem~\ref{open_cor}) 
and characterize the type-$\I$ property of square-integrable representations in terms of coadjoint orbits (Theorem~\ref{F_gen} and Corollary~\ref{open_cor2}). 
We provide some information on the topological properties of
the Puk\'anszky correspondence for (connected, simply connected) solvable Lie groups $G$. 
Namely, we  prove that to every open coadjoint orbit $\Oc$ of $G$ there corresponds a compact open subset of $\Prim(G)$
which is homeomorphic to a torus which is dual to the fundamental group of $\Oc$,
in the sense of the duality theory of locally compact abelian groups (see Theorem~\ref{homeo}.)

In Theorem~\ref{codim1_th}, we improve a result from the earlier literature \cite[Th. 4.5]{KT96}. 
Specifically, we prove that if $G$ is a solvable Lie group whose nilradical is 1-codimensional, then the set
of 
the isolated points of  the primitive ideal spectrum $\Prim(G)$ of the group $C^*$-algebra $C^*(G)$
coincides with the set of 
the kernels in $C^*(G)$ of the square-integrable unitary irreducible representations of $G$.
We recall from \cite{Gr80} that the isolated points of $\Prim(G)$ are precisely the kernels of the square-integrable factor representations of $G$. 
Our result says that, under the additional hypothesis on the nilradical, every square-integrable factor representation of $G$ is 
type~$\I$, hence, quasi-equivalent to an irreducible representation. 
We prove by example in Section~\ref{sect7} that the hypothesis on the nilradical cannot be omitted. 

The structure of this paper is as follows: In Section~\ref{Prel} we collect some basic facts on the Puk\'anszky correspondence between the generalized orbits of the coadjoint action  and the quasi-equivalence classes of normal representations of solvable Lie groups. We also show that the elements of the open coadjoint 
orbits are always in general position, in the sense that their coadjoint orbits have maximal dimension, 
and the number of open simply connected coadjoint orbits is even. 
Section~\ref{sect3} contains the intrinsic characterization of square-integrable factor representations in terms of the generalized orbits of the coadjoint action. 
In Section~\ref{sect4}  we show that there are no clopen points in the primitive ideal space of a unimodular  (connected, simply connected) solvable Lie group, therefore these groups do not have square integrable representations. 
In Section~\ref{sect5} 
we discuss the type-$\I$ property of representations corresponding to open coadjoint orbits (Theorem~\ref{open_cor})
and show that there is a bijective correspondence between  the type~$\I$ square-integrable representations 
and the simply connected open coadjoint orbits of the group. 
In addition, we prove that the existence of an open coadjoint orbit implies the existence of a compact open subset of the primitive ideal space. 
(Corollary~\ref{open_cor2}).  
We show  in Section~\ref{sect6} that this property is shared by all square-integrable factor representations of solvable Lie groups whose  nilradical has codimension~1 (Theorem~\ref{codim1_th}).  
The examples in Section~\ref{sect7} prove that this is not the case when the nilradical has codimension at least $3$.

\subsection*{General notation}
For every topological group $K$ we denote by $K_\1$ the connected component of the unit element $\1\in K$.

The 1-connected (that is, connected and simply connected)  Lie groups are denoted by upper case Roman letters and their Lie algebras by the corresponding
lower case Gothic letters. 
We are mainly interested in solvable/nilpotent Lie groups that are
1-connected, and the exceptions will be particularly emphasized.  
Every connected closed subgroup of a 1-connected solvable Lie group is 1-connected by \cite[Prop. 11.2.15]{HN12}.

An exponential Lie group is a Lie group $G$ whose exponential map
$\exp_G \colon\gg \to G$ is bijective. 
All exponential Lie groups are solvable. 
See for instance \cite{ArCu20} for more details. 

For any Lie algebra $\gg$ with its linear dual space $\gg^*$ we denote by $\langle\cdot,\cdot\rangle\colon\gg^*\times\gg\to\RR$ the corresponding duality pairing. 
We often denote the group actions either by $\cdot$ or by juxtaposition,  in particular for the coadjoint action $G\times\gg^*\to\gg^*$, $(g,\xi)\mapsto g\xi$.

\section{Preliminaries on generalized orbits of the coadjoint action}
\label{Prel}

\subsection{Puk\'anszky theory for solvable Lie groups}
We collect here a few elements of the Puk\'anszky theory on representations of solvable Lie groups that are needed later on, and we include precise references to Puk\'anszky's original papers 
for proofs. 
(See also \cite{BB21b}.)

Unless otherwise mentioned, $G$ denotes a 1-connected solvable Lie group with its Lie algebra $\gg$.
The Lie algebra of the connected normal subgroup 
$D:=[G,G]$ 
of $G$ is 
the derived ideal $\dg:=[\gg,\gg]$. 
We denote by $\iota\colon\dg\hookrightarrow\gg$ the inclusion map, 
and let $\iota^*\colon\gg^*\to\dg^*$, $\xi\mapsto\xi\vert_\dg$ be its dual map.

We fix a 1-connected solvable Lie group $\Galg$ with its Lie algebra $\ggalg$ for which $G\subseteq \Galg$ is a closed subgroup, $[\gg,\gg]=[\ggalg,\ggalg]$, and $\ggalg$ is isomorphic to an algebraic Lie algebra. 
Note that such a group always exists, but it is not unique and not canonical.
For instance, $\Galg\times\mathbb{R}^k$ shares all the above properties of $\Galg$ for arbitrary $k\ge1$.
See \cite[page 521]{Pu71}, where $\Galg$ is denoted by $\widetilde{G}$. 

Let \Rel\  be 
the equivalence relation on $\gg^*$ defined  by 
 $$\xi \text{ is \Rel-equivalent to $\eta$  } \iff \overline{G\xi}=\overline{G\eta}, $$
 where  $\xi,\eta\in\gg^*.$
The equivalence classes of \Rel\ are called \emph{quasi-orbits} of the coadjoint action and their set is denoted by $(\gg^*/G)^\sim$. 
The corresponding quasi-orbit map is 
$$r\colon \gg^*\to(\gg^*/G)^\sim, \quad r(\xi)=\{\eta\in\gg^*\mid \overline{G\xi}=\overline{G\eta}\}.$$
Then
\begin{equation}\label{quasi_eq1}
G\xi\subseteq r(\xi)\subseteq\overline{G\xi}\text{ for all }\xi\in\gg^*.
\end{equation}
In fact \Rel\ is the only equivalence relation on~$\gg^*$ for which each equivalence class is locally closed, $G$-invariant, and contained in the closure of the coadjoint orbit of each of its elements. 
Moreover, if  $\Oc\in(\gg^*/G)^\sim$, then 
there exists a connected closed subgroup~$G_1$ of $\Galg$ with  $G\subseteq G_1\subseteq \Galg$, 
such that 
\begin{equation}
\label{open_qo_proof_eq1}
\Oc=G_1\xi=\overline{G\xi}\cap\Galg\xi.
\end{equation}
for every $\xi\in\Oc$. 
Here we note that, since $[\gg,\gg]=[\ggalg,\ggalg]$ and 
$G\subseteq G_1\subseteq \Galg$ are connected closed subgroups, 
we have $[\Galg,\Galg]=[G_1,G_1]=[G,G]=D\subseteq G$. 
Therefore $G$ is a normal subgroup of $G_1$, and then the restriction map $\gg_1^*\to\gg^*$ 
intertwines the coadjoint action of $G_1$ with  a natural 
smooth action of $G_1$ on $\gg^*$, which is involved in the equality $\Oc=G_1\xi$ in~\eqref{open_qo_proof_eq1}. 
Since the Lie group action $G_1\times\gg^*\to\gg^*$ is smooth, we see that $\Oc$ is endowed with a $G_1$-homogeneous space structure whose inclusion map into $\gg^*$ is an immersion, 
cf. \cite[Prop. 10.1.14]{HN12}. 
Thus, every coadjoint quasi-orbit of $G$ has the structure of a smooth manifold immersed into $\gg^*$. 
(See \cite[Lemma 1]{Pu86}
  and its proof for more details.)
We also note that 
  \begin{equation*}
	\Oc \in\gg^*/G \, \iff \, 	G\xi\subseteq \gg^* 
	\text{ is locally closed for some }\xi \in \Oc.
  \end{equation*}

\subsection*{Torus bundles over coadjoint quasi-orbits}
We recall the following notation, cf. \cite[pages 491--492]{Pu71}.
Let $\xi\in\gg^*$ be arbitrary fixed.
\begin{itemize}
\item 	$G(\xi):=\{g\in G\mid g\xi=\xi\}$ is the stabilizer at $\xi$ of the coadjoint action,  and
	 $G(\xi)_\1$ denotes the connected component of $\1\in G(\xi)$.
	The Lie algebra of $G(\xi)$ is 
	$\gg(\xi):=\{x\in\gg\mid(\forall y\in\gg)\  \langle\xi,[x,y]\rangle=0\}$,
	\item Let $\chi_\xi\colon G(\xi)_\1\to\TT$ be the character defined by
	$\chi_\xi(\exp_G x)=\ee^{\ie \langle\xi,x\rangle}$. 
	The reduced stabilizer at $\xi$ is 
	$$\overline{G}(\xi):=\{g\in G(\xi)\mid (\forall h\in G(\xi))\ ghg^{-1}h^{-1}\in\Ker\chi_\xi\}.$$
	\end{itemize}
	By \cite[Cor.~4.1 and proof of Lemma~4.1, page 492]{Pu71}, 
	\begin{equation}\label{fin-gen}
G(\xi)/G(\xi)_\1 \; \; \text{is a finitely generated free abelian group},
\end{equation}
	hence $\overline{G}(\xi)/G(\xi)_1$ is a finitely generated free abelian group, as well.
	There is an integer $\rk(\xi)\in\NN$  such that  $\overline{G}(\xi)/G(\xi)_\1\simeq\ZZ^{\rk(\xi)}$
	as groups. 
Actually, for every quasi-orbit $\Oc\in(\gg^*/G)^\sim$  and $\xi \in \Oc$, the number  $\rk(\xi)$
 depends  on $\Oc$ only, so we may sometimes write
$\rk(\Oc)=\rk(\xi)$,  $\xi\in\Oc$.
	Denote $$\mathop{G}\limits^{\mathsout{\wedge}}(\xi):=
	\{\chi\in\Hom(\overline{G}(\xi),\TT)\mid \chi\vert_{G(\xi)_\1}=\chi_\xi \}.$$
The definition of $\overline{G}(\xi)$ shows that 
\begin{equation}\label{prel1}
[G(\xi), \overline{G}(\xi)]\subseteq \Ker \chi_\xi\subseteq G(\xi)_1\subseteq  \overline{G}(\xi).
\end{equation}
When $G(\xi)$ is abelian,  $\overline{G}(\xi)= G(\xi)$.

The set 
$$\Bun(\gg^*):=\bigsqcup_{\xi\in\gg^*}\{\xi\}\times \mathop{G}\limits^{\mathsout{\wedge}}(\xi) 
$$
carries the natural action 
$$G\times \Bun(\gg^*)\to\Bun(\gg^*), \quad (g,(\xi,\chi))\mapsto (g\xi,g\chi).$$ 
Indeed, 
$g^{-1}\overline{G}(g\xi) g= \overline{G}(\xi)$  for every  $\xi\in\gg^*$ and $g\in G$.
Hence, for every
$\chi\in \mathop{G}\limits^{\mathsout{\wedge}}(\xi)$  
there is a character 
 $g\chi\in \mathop{G}\limits^{\mathsout{\wedge}}(g\xi)$ 
defined by  $(g\chi)(h)=\chi(g^{-1}hg)$ for all $h\in \overline{G}(g\xi)$.

\begin{lemma}\label{extra-lemma}
For every $\xi \in \gg^*$, $g \in G(\xi)$, and $\chi \in \mathop{G}\limits^{\mathsout{\wedge}}(\xi)$ we have
\begin{equation}\label{invar}
g \chi =\chi.
\end{equation}
\end{lemma}

\begin{proof}
By \eqref{prel1}, 
the action of $g \in G(\xi)$ leaves $\mathop{G}\limits^{\mathsout{\wedge}}(\xi)$
invariant.
For $h \in \overline{G}(\xi)$ we write 
$$(g\chi)(h)=\chi(g^{-1}hg)=\chi (g^{-1}hgh^{-1})\chi(h).$$
Here $\chi (g^{-1}hgh^{-1})=1$ since $g\in G(\xi)$ and  $h\in \overline{G}(\xi)$,
hence $(g\chi)(h)=\chi(h)$.
\end{proof}

The group $X(G)\simeq \widehat{G/D}$ acts
on $\Bun(\gg^*)$ by $X(G) \times \Bun(\gg^*) \to \Bun (\gg^*)$, $ (\varphi, p)\mapsto \varphi p$, 
where 
\begin{equation*}
 \varphi p= (\sigma+\xi, \varphi\vert_{\overline{G}(\xi)}\chi)
 \end{equation*}
for all  
$p=(\xi, \chi) \in \Bun(\gg^*)$, $\varphi \in X(G)$, and where $\sigma\in \dg^\perp$ is defined by
$\ie \sigma = d\varphi$.

For a subset $\Xi\subseteq\gg^*$ we denote  
$$\Bun(\Xi):=\tau^{-1}(\Xi)=\bigsqcup_{\xi\in\Xi}\{\xi\}\times \mathop{G}\limits^{\mathsout{\wedge}}(\xi), 
$$
where 
$$\tau\colon \Bun(\gg^*)\to\gg^*, \quad \tau(\xi,\chi):=\xi.$$ 
For every coadjoint quasi-orbit $\Oc\in(\gg^*/G)^\sim$ of $G$
there exists a $G$-equivariant bijection 
$\Bun(\Oc)\to \Oc\times\TT^{\rk(\Oc)}$ such that the diagram 
\begin{equation}\label{alpha}
\xymatrix{
\Bun(\Oc) \ar[dr]_{\tau\vert_{\Bun(\Oc)}} \ar[rr] &  &
\Oc\times\TT^{\rk(\Oc)} \ar[dl]^{\pr_1}\\ 
& 
\Oc &
}
\end{equation}
is commutative, 
where $\pr_1$ stands for the Cartesian projection onto the first factor. 
(See \cite[Subsect. 6.3, page 537]{Pu71}.)
Then  $\Bun(\Oc)$ is endowed with the smooth manifold structure transported from 
$\Oc\times\TT^{\rk(\Oc)}$ via the above bijection. 

For later use, we sketch an equivalent description of the smooth structure of $\Bun(\Oc)$ given in 
\cite[end of \S III.1(3), page 825]{Pu86}.
Define $K:=D\overline{G}(\xi)$; this a closed subgroup of $G$, independent on $\xi\in\Oc$, 
and the connected component of $\1\in K$ is $K_\1=DG(\xi)_\1$. 
Let $q\colon K\to K/K_\1=:\Pi$ be the corresponding quotient map. 
Then $\Pi$ is a finitely generated free abelian group, 
isomorphic with $\overline{G}(\xi)/G(\xi)_\1$. 
The group 
$\widetilde{\Pi}:=\{\phi\in \Hom(K,\TT)\mid \phi\vert_{K_1} \equiv 1\}\simeq \Hom(\Pi,\TT)$  is the structural group of the principal bundle $\tau\vert_{\Bun(\Oc)}\colon \Bun(\Oc)\to\Oc$. 
For every $a\in \Pi$, there is a smooth mapping $b\colon \Oc \to K$ such that $b(\xi) \in \overline{G}(\xi)$ for every $\xi\in \Oc$ and $q(b(\xi))=a$ for every $\xi \in \Oc$. 
Such a mapping $b$ is called \textit{admissible}.
For an admissible $b$, we define 
$$f_b\colon\Bun(\Oc)\to\TT,\quad f_b(\xi,\chi):=\chi(b(\xi)).$$
With this notation, the smooth manifold structure of $\Bun(\Oc)$ is uniquely determined by the condition 
$f_b\in\Ci(\Bun(\Oc),\TT)$ for every admissible function $b$. 

\subsection*{The Puk\'anszky correspondence} 
For every coadjoint quasi-orbit $\Oc\in(\gg^*/G)^\sim$ 
its corresponding bundle $\Bun(\Oc)$ carries the action of $G$.  
The orbit closures of that action  constitute a partition of $\Bun(\Oc)$ by \cite[Prop. 7.1, page 539]{Pu71}. 
We denote by
\begin{equation}\label{bun_orbits}
(\Bun(\Oc)/G)^\approx=\{\overline{Gp}\mid p=(\xi,\chi)\in\Bun(\Oc)\}
\end{equation}
the set of these orbit closures.
Thus
$$\RelS:=\bigcup\limits_{\Oc\in (\gg^*/G)^\sim}(\Bun(\Oc)/G)^\approx,$$ 
is a set of subsets  of $\Bun(\gg^*)$, and these subsets constitute a partition of $\Bun(\gg^*)$.
Each such a subset  $\O\in\RelS$ is called a \emph{generalized orbit of the coadjoint action of~$G$}.

There is a bijective correspondence between the elements of $\RelS$ and the primitive ideals of $C^*(G)$ (see \cite{Pu73}), and each primitive ideal of $C^*(G)$ is the kernel of a unique quasi-equivalence class of normal representations, so that there is bijective mapping, the Puk\'anszky correspondence, 
 $$\ell\colon \RelS\to \stackrel{\frown}{G}_\nor.$$ 
 (See \cite[Thm. 3, page 134]{Pu74}.)
 
We give a brief description of the correspondence between $\RelS$ and $\Prim(C^*(G))$. 
Let  $p=(\xi, \chi) \in \Bun(\gg^*)$ with $\xi\in\gg^*\setminus \dg^\perp$  and $\chi\in\mathop{G}\limits^{\mathsout{\wedge}}(\xi)$,  
and  let $\pi:=\kappa_D(D\xi\vert_{\dg})$, where
$\kappa_D\colon \dg^*/D\to \widehat{D} $ denotes the Kirillov homeomorphism for the nilpotent Lie group~$D$.
There exists $\lambda(p)\in \widehat{K}$ (where $K=D\overline{G}(\xi)=\overline{G}(\xi)D$ as above) such that 
$\lambda(p) \vert_{D}=\pi$ and 
\begin{align}
& \lambda (a p) = \lambda (p) \; \; \text{ for every } a \in D G(\xi\vert_{D})_\1 ,  \label{1001} \\
& \lambda (\varphi p) =\varphi \vert_{K} \lambda(p) \; \text { for every } \varphi\in X(G). \label{1002}
\end{align}
(See \cite[c1) and c2), pages 83--84]{Pu73}.)
Then  $T(p):=\Ind_K^G \lambda(p)$ is a factor representations of $G$, whose quasi-equivalence class depends only on the generalized orbit 
$\O=\overline{Gp}\in (\Bun(\Oc)/G)^\approx$, 
where $\Oc = r(\xi)$. 
The map
$$
 \Jc\colon  \Bun(\gg^*)\to \Prim(G), \; \; \Jc(p)=\Ker T(p)
 $$
is constant on generalized orbits, that is, 
$\Jc(p_1)=\Jc(p_2)=: J(\O)$ if $ p_1, p_2\in \O$ 
with $\O\in (\Bun(\Oc)/G)^\approx$
for some $\Oc \in (\gg^*/G)^\sim$. 
 Then
 \begin{equation}\label{Jc}
 \Jc \vert_{\Bun(\Oc)} \; \text{is continuous}
\end{equation}
(see \cite[Lemma~8, page 93]{Pu73}), 
the map 
\begin{equation}\label{J}
\RelS\to \Prim(G), \quad \O \mapsto  J(\O), 
\end{equation}
is a bijection, and 
$ J(\O)=\Ker \ell(\O)$.
(See  \cite[Thm.~1, page 114]{Pu73}.)

Finally, we recall that
\begin{equation}\label{open_cor1_proof_eq1}
	\ell(\O)\text{ is type~$\I$} \iff 
	\O\in\Bun(\Oc)/G \text{ and }
	[G(\xi):\overline{G}(\xi)]<\infty\text{ for some }\xi\in\Oc.
	\end{equation}
	(See \cite[Thm. 2, pages 551--552]{Pu71}.)
	
	\begin{lemma}\label{fixed}
If $H$ is a connected Lie group, $\sigma\in[\hg,\hg]^\perp\subseteq\hg^*$, and $h\in H$, then $h\sigma=\sigma$. 
\end{lemma}

\begin{proof}
Since the Lie group $H$ is connected, we have 
$$H=\bigcup\limits_{n\ge 1}\underbrace{\exp_H(\hg)\cdots\exp_H(\hg)}_{n\ \rm times}.$$
Therefore it suffices to show that for every $x\in\hg$ we have $(\exp_Hx)\sigma=\sigma$.  
To this end we define $\varphi\colon\RR\to \hg^*$, $\varphi(t):=(\exp_H(tx))\sigma$. 
Then $\varphi\in\Ci(\RR,\hg^*)$ and for arbitrary $t_0\in\RR$ we have 
\begin{align*}
\varphi'(t_0)
&=\lim_{t\to t_0}\frac{1}{t}((\exp_H(tx))\sigma-(\exp_H(t_0x))\sigma) \\
&=(\exp_H(t_0x))\lim_{t\to 0}\frac{1}{t}((\exp_H(tx))\sigma-\sigma) \\
&=(\exp_H(t_0x))(\sigma\circ\ad_\hg x) \\
&=0,
\end{align*}
since $\sigma\in[\hg,\hg]^\perp$, hence $\sigma\circ\ad_\hg x=0$. 
Therefore the function $\varphi\colon\RR\to \hg^*$ is constant, 
and then $(\exp_Hx)\sigma=\varphi(1)=\varphi(0)=\sigma$, which completes the proof. 
\end{proof}

\begin{remark}\label{B_orbits}
\normalfont
Since $\dg=[\ggalg, \ggalg]$, we have $a\sigma=\sigma$ for every $\sigma\in\dg^\perp$ and $a\in\Galg$, as we will see in Lemma~~\ref{fixed}.
This implies that the action of the additive group $\dg^\perp$  on $\gg^*$ by translations commutes with the action of $\Galg$, 
hence one has the group action of the product group $(\Galg\times\dg^\perp)\times\gg^*\to\gg^*$, 
$((a,\sigma),\xi)\mapsto a\xi+\sigma$. 
For every $(\Galg\times\dg^\perp)$-orbit $\Omega\in\gg^*/(\Galg\times\dg^\perp)$ the following assertions hold: 
\begin{enumerate}[{\rm(i)}]
	\item\label{B_orbits_item1} The subset $\Omega\subseteq\gg^*$ is locally closed. 
	\item\label{B_orbits_item2} The set $\Bun(\Omega)$ has the structure of a smooth manifold  
	for which the mapping $\tau\vert_{\Bun(\Omega)}\colon\Bun(\Omega)\to\Omega$ is a trivial principal trivial bundle whose structural group is $\TT^{\rk(\xi)}$ for any $\xi\in\Omega$. 
	\item If $\Oc\in(\gg^*/G)^\sim$ and $\Oc\subseteq\Omega$, then for all $\xi \in \Oc$, we have
	$ G\xi \subseteq \Oc \subseteq \Galg\xi \subseteq \Omega$, and
	the principal bundle $\Bun(\Oc)$ is isomorphic to the restriction of the principal bundle $\Bun(\Omega)$ 
	from $\Omega$ to $\Oc$. 
	\item\label{B_orbits_item4}  For each $\xi\in\Omega$, the orbit $\Galg\xi$ is closed in $\Omega$.
	\end{enumerate}
See \cite[Subsect. 5, page 85]{Pu73} and \cite[Subsect. 5, page 529]{Pu71} for these facts. 
The relation between the bundles $\Bun(\Oc)$ and $\Bun(\Omega)$ is discussed in \cite[paragraph~c., page 538]{Pu71}, 
while \eqref{B_orbits_item4} follows from \cite[Lemma 3.2.1, page 553]{Pu71}.
\end{remark}

\subsection{Open coadjoint quasi-orbits}\label{open-quasi}
We analyse here the main geometric object we are concerned with, the open coadjoint quasi-orbits. 

For the  solvable Lie group $G$, 
we denote $k_0:=\min\{\dim\gg(\xi)\mid\xi\in\gg^*\}$.
Then let
$$\gg^*_{\rm gen}:=\{\xi\in\gg^*\mid\dim\gg(\xi)=k_0\}, $$
be the union of all coadjoint orbits of $G$ having maximal dimension.

\begin{proposition}\label{minimal}
	If $\Oc\in(\gg^*/G)^\sim$ is open in $\gg^*$, then $\Oc\subseteq\gg^*_{\rm gen}$. 
\end{proposition}

\begin{proof}
Let us consider the linear operator $T\colon\gg^*\to\Lc(\gg,\gg^*)$, 
where for every $\xi\in\gg^*$ we define $T\xi\colon\gg\to\gg^*$ by $(T\xi)(x):=(\ad_\gg^*x)\xi=-\xi\circ\ad_\gg x$, 
hence $\Ker (T\xi)=\gg(\xi)$. 
Then, denoting $m:=\dim\gg$, we have 
$m-k_0=\max\{\dim(\Ran T\xi)\mid\xi\in\gg^*\}$ and 
$\gg^*_{\rm gen}=\{\xi\in\gg^*\mid \dim(\Ran T\xi)=m-k_0\}$. 
Then  $\gg^*_{\rm gen}$ is a Zariski open subset of $\gg^*$, and in particular $\gg^*_{\rm gen}$ is a dense open subset of~$\gg^*$. 

Since $\Oc\subseteq\gg^*$ is an open subset, it follows that  $\Oc\cap\gg^*_{\rm gen}\ne\emptyset$. 
For arbitrary $\xi\in\Oc$ we now prove by contradiction that $\xi\in\gg^*_{\rm gen}$. 
Indeed, if we assume that $\xi\in\gg^*\setminus\gg^*_{\rm gen}$ then,  
since $\gg^*_{\rm gen}$ is an open $G$-invariant subset of $\gg^*$, 
we obtain $\overline{G\xi}\subseteq \gg^*\setminus\gg^*_{\rm gen}$. 
Recalling that $\Oc=r(\xi)\subseteq\overline{G\xi}$ 
by \eqref{quasi_eq1}, we then obtain $\Oc\subseteq \gg^*\setminus\gg^*_{\rm gen}$, which is a contradiction with $\Oc\cap\gg^*_{\rm gen}\ne\emptyset$. 
Thus $\Oc\subseteq \gg^*_{\rm gen}$, and this completes the proof. 
\end{proof}

\begin{corollary}
\label{minimal_cor}
If $G$ has open coadjoint orbits then every open coadjoint quasi-orbit is actually a coadjoint orbit. 
\end{corollary}

\begin{proof}
All open coadjoint orbits of $G$
are contained in 
$\gg^*_{\rm gen}$ by  Proposition~\ref{minimal}, and  have the same dimension as  $\gg^*$. 
Thus  $\dim G\eta=\dim \gg^*$, for every  $\eta\in  \gg^*_{\rm gen}$.
It follows, by  \cite[Ch. III, \S 1, no.~7, Rem.]{Bo06}, that $G\eta$ is open for every  $\eta\in  \gg^*_{\rm gen}$
that is, the union of all open coadjoint orbits of $G$ is 
equal to~$\gg^*_{\rm gen}$. 

For an arbitrary open coadjoint quasi-orbit  $\Oc\in(\gg^*/G)^\sim$, 
Proposition~\ref{minimal} shows that$\Oc\subseteq\gg^*_{\rm gen}$. 
Hence, by the above paragraph, there exists $\xi\in\Oc$ whose coadjoint orbit $G\xi\subseteq\gg^*$ is open in $\gg^*$, and in particular is locally closed. 
This implies $\Oc=G\xi$, 
and we are done. 
\end{proof}

The next proposition extends \cite[Prop.~6.1]{BB18a} from exponential Lie groups to arbitrary 1-connected solvable Lie groups.

\begin{proposition}\label{even}
For every 1-connected solvable Lie group $G$ with $\dim G>0$, 
the number of open simply connected coadjoint orbits of $G$ 
is finite and even.
\end{proposition}

\begin{proof} 
Let $\Fc\subseteq\gg^*/G$ be set of simply connected open coadjoint orbits of~$G$. 
Then $\Fc$ is a finite set by \cite[Prop.~4.5 (ii)]{BB16}.

Let $\Oc\in \Fc$ and select any  $\xi \in \Oc$, hence $\Oc = G\xi=:\Oc_\xi$. 
We claim that  there is
$\xi_0\in \Oc$ such that $-\xi_0\not\in \Oc$.

Assume the contrary; then the map $\omega_0 \colon \Oc  \to \Oc$, $\omega_0(\xi)=-\xi$, is well-defined and continuous, since $\Oc$ is open  and the map $\gg^*\to \gg^*$, $\eta\mapsto -\eta$ is continuous. 
Moreover, $\omega_0\circ \omega_0= \text{id}$. 
Since $\Oc$ is simply connected and $G$ is a 1-connected solvable Lie group, 
there is a diffeomorphism 
$\chi\colon \Oc \to \RR^d$, where $d=\dim \Oc$. 
(See e.g., \cite[Thm.~2]{Pu92}.) 
The map $\omega \colon \RR^d \to \RR^d$, $\omega= \chi\circ \omega_0\circ \chi^{-1}$ is a homeomorphism such that $\omega\circ \omega= \text{id}$.
Then,  by P.A~Smith's fixed point theorem (see \cite[Thm.~1a]{Ei40}), the set of fixed points of $\omega$
is not empty, hence there is $\eta\in \Oc$ such that $-\eta=\omega_0(\eta)=\eta$.
It follows that $0=\eta\in \Oc$, that is, $\Oc =\{0\}$, and this is a contradiction. 

We have thus proved that there is $\xi_0\in \Oc$ such that $-\xi_0\not \in \Oc$, hence 
$$\Oc=\Oc_{\xi_0}\ne \Oc_{-\xi_0}=-\Oc.$$
On the other hand, the map $\Oc=\Oc_{\xi_0} \to \Oc_{-\xi_0}=-\Oc$, $ \xi \mapsto -\xi$  is a homeomorphism. 
Thus, for every $\Oc\in\Fc$ we have $-\Oc\in\Fc$ and $-\Oc\ne\Oc$. 
Hence the finite set $\Fc$ is the disjoint union of 2-element sets $\{\Oc,-\Oc\}$ with $\Oc\in\Fc$, 
which shows that the number of elements of $\Fc$ is even. 
\end{proof}

\section{Square integrable representations and the corresponding generalized coadjoint orbits}
\label{sect3}

Our key technical result in this paper is the following characterization of square-integrable factor representations in terms of geometric objects that are intrinsically associated to the solvable Lie group under consideration. 

\begin{theorem}\label{open_th}
Let $G$ be a $1$-connected solvable Lie group. 
Assume  that $\pi$ is a normal representation of $G$,  and let
  $\O\in (\Bun(\Oc)/G)^\approx$, for a coadjoint quasi-orbit $\Oc\in (\gg^*/G)^\sim$, be the generalized orbit associated to the 
  quasi-equivalence class of $\pi$ by the Puk\'ansky correspondence. 
Then the following assertions are equivalent: 
\begin{enumerate}[{\rm(i)}]
	\item\label{open_th_item1} The representation $\pi$  is square integrable. 
	\item\label{open_th_item2} The group $G$ has trivial centre, the subset $\Oc\subseteq\gg^*$ is open, 
	and every $G$-orbit in $\Bun(\Oc)$ is dense. 
\end{enumerate}
\end{theorem}

We note that the condition
that every $G$-orbit in $\Bun(\Oc)$ is dense can be written equivalently as $\O=\Bun(\Oc)$, or
$(\Bun(\Oc)/G)^\approx=\{\Bun(\Oc)\}$, or $\Bun(\Oc)\in(\Bun(\Oc)/G)^\approx$. 

The proof of Theorem~\ref{open_th} requires some preparation. 
Our starting point is the following characterization of square integrability obtained
in \cite[Thm. 3.4]{Ros78}, 
which is not intrinsic, as it depends on the non-canonical auxiliary group $\Galg$ introduced in the beginning of Section~\ref{Prel}. 

\begin{lemma}\label{open_lemma}
Assume the group $G$  has trivial centre. 
Let $\Oc\in (\gg^*/G)^\sim$, and $\O\in (\Bun(\Oc)/G)^\approx$. 
Select any $\xi\in\Oc$, denote $\xi_0:=\xi\vert_\dg\in\dg^*$, 
and define 
$\Omega:=\dg^\perp+\Galg\xi\subseteq\gg^*$. 
Then the following assertions are equivalent: 
\begin{enumerate}[{\rm(i)}]
	\item\label{open_lemma_item1} 
	The quasi-equivalence class $\ell(\O)\in \stackrel{\frown}{G}_\nor$ is square integrable. 
	\item\label{open_lemma_item2} 
	The subset $\Galg\xi_0\subseteq\dg^*$ is open
	and for every $p\in\Bun(\Omega)$ the subset $Gp\subseteq\Bun(\Omega)$ is dense. 
\end{enumerate}
\end{lemma}

\begin{proof}
Specialize \cite[Thm. 3.4]{Ros78} for the case when the centre of $G$ is $Z=\{\1\}$. 
\end{proof}

We aim to replace the above conditions involving $\Galg$ by conditions depending on~$G$ only. 
To do this, we need some lemmas.

Stability properties of open orbits, as in the following result, have appeared in other contexts in the earlier literature. 
See for instance \cite[Lemma 3.1]{Fu15}.

\begin{lemma}\label{open_orbit}
	Let $H$ be a connected Lie group with a closed connected subgroup $K\subseteq H$ satisfying $[\kg,\kg]=[\hg,\hg]$. 
	 If $\xi\in\kg^*$ and $H\xi\subseteq\kg^*$ is open, 
	 then $[\kg,\kg]^\perp+H\xi=H\xi$.
\end{lemma}

\begin{proof}
Let $\Vert\cdot\Vert$ be any norm on $\kg^*$. 
Since $H\xi\subseteq\kg^*$ is open, there exists $\varepsilon_0>0$ such that for every $\eta\in\kg^*$ with $\Vert\eta\Vert<\varepsilon_0$ we have $\eta+\xi\in H\xi$. 

Now let $\sigma\in[\kg,\kg]^\perp\subseteq\kg^*$ be arbitrary with $\Vert\sigma\Vert<\varepsilon_0$. 
To conclude the proof, we show by induction that for every integer $n\ge 1$ we have $n\sigma+\xi\in H\xi$. 
In fact, for $n=1$, this follows by the way $\varepsilon_0$ was chosen. 
If we assume the assertion holds true for $1,\dots,n$, then there exist $h_1,\dots,h_n$ with $j\sigma+\xi=h_j\xi$ for $j=1,\dots,n$. 
Then 
$$(n+1)\sigma+\xi=\sigma+(n\sigma+\xi)=\sigma+h_n\xi=h_n(\sigma+\xi)=h_nh_1\xi$$
where we used the equality $\sigma=h_n\sigma$, which follows by Lemma~\ref{fixed}. 
Since $h_nh_1\in H$, this completes the induction step, and we are done. 
\end{proof}

In the next two lemmas and in the proof of Theorem~\ref{open_th} below 
we use again the non-canonical auxiliary group $\Galg$ introduced in the beginning of Section~\ref{Prel}. 

\begin{lemma}\label{open_qo}
If  $\Oc\in(\gg^*/G)^\sim$ and $\Oc\subseteq\gg^*$ is open, then 
$\Oc\in\gg^*/\Galg$ and $\Oc=\Oc+\dg^\perp$. 
Moreover, if $\xi\in\Oc$ and we denote $\xi_0:=\xi\vert_\dg\in\dg^*$, then 
the subset $\Galg\xi_0\subseteq\dg^*$ is open. 	
\end{lemma}

\begin{proof}
Let $G_1$ be a connected closed subgroup of $\Galg$ such that $G\subseteq G_1\subseteq \Galg$, as in~\eqref{open_qo_proof_eq1}.
For arbitrary $\xi\in\Oc$, it follows by \eqref{open_qo_proof_eq1} that 
$\Oc$ 
is a relatively closed subset of~$\Galg\xi$. 
But $\Oc\subseteq\gg^*$ is open, hence $\Oc$ is an open subset of $\Galg\xi$ as well. 
Since the group $\Galg$ is connected, 
its orbit $\Galg\xi$ is connected, 
and we then obtain $\Oc=\Galg\xi$. 
Moreover, $G_1\xi=\Oc\subseteq\gg^*$ is open,  
hence by Lemma~\ref{open_orbit} we have 
$\Oc=\Oc+\dg^\perp$.

On the other hand, since the subset $\Galg\xi\subseteq\gg^*$ is open and the restriction map $\iota^*\colon\gg^*\to\dg^*$ is an open mapping, it follows that the subset $\iota^*(\Galg\xi)\subseteq\dg^*$ is open. 
Since $\iota^*(\Galg\xi)=\Galg\iota^*(\xi)=\Galg\xi_0$, 
we thus see that $\Galg\xi_0\subseteq\dg^*$ is open.
\end{proof}

\begin{lemma}\label{qo}
If $\Omega\in\gg^*/(\Galg\times\dg^\perp)$, then  $\Omega\in(\gg^*/G)^\sim$ if and only if $\Omega\subseteq\overline{G\eta}$ for every $\eta\in\Omega$. 
\end{lemma}

\begin{proof}
If $\Omega\in(\gg^*/G)^\sim$ then for every $\eta\in\Omega$ we have $\Omega=r(\eta)\subseteq\overline{G\eta}$ by \eqref{quasi_eq1}. 
	
For the converse implication, note first that 
since  $\Omega\in\gg^*/(\Galg\times\dg^\perp)$,
the subset $\Omega\subseteq\gg^*$ is locally closed by Remark~\ref{B_orbits}\eqref{B_orbits_item1}. 
Moreover, $G\Omega=(G\times\{0\})(\Galg\times\dg^\perp)\Omega
=(\Galg\times\dg^\perp)\Omega\subseteq\Omega$. 
Thus for every $\eta\in\Omega$ we have $G\eta\subseteq\Omega$, 
while $\Omega\subseteq\overline{G\eta}$ by hypothesis. 
It follows that $\overline{G\eta}=\overline{\Omega}$ for arbitrary $\eta\in\Omega$, 
which shows that the points of $\Omega$ are \Rel-equivalent to each other. 
That is, $\Omega\subseteq r(\eta)$ for every $\eta\in\Omega$.  

If $\eta\in\Omega$ and $\zeta\in r(\eta)$, then $r(\zeta)=r(\eta)$ hence, by \eqref{quasi_eq1}, 
we have  $\overline{G\zeta}=\overline{r(\zeta)}=\overline{r(\eta)}=\overline{\Omega}$. 
Since $\Omega$ is locally closed in $\gg^*$, it follows that $\Omega$ is a relatively open subset of $\overline{\Omega}$, while on the other hand $G\zeta$ is a dense subset of $\overline{G\zeta}=
\overline{\Omega}$, 
hence $\Omega\cap G\zeta\ne\emptyset$. 
Since $\Omega$ is $G$ invariant, we then obtain $\zeta\in\Omega$. 
Thus $r(\eta)\subseteq\Omega$, and then $\Omega= r(\eta)\in(\gg^*/G)^\sim$. 
\end{proof}

\begin{proof}[Proof of Theorem~\ref{open_th}]
``\eqref{open_th_item1}$\implies$\eqref{open_th_item2}'' 
Since $G$ has square-integrable representations, its centre $Z$ is compact, 
as noted in \cite[\S 2.2]{Ros78}. 
Specifically in the present situation, the fact that $\pi\colon G\to\Bc(\Hc_\pi)$ is square integrable means, by definition, that $\pi$ is a factor representation and there exist $v,w\in\Hc_\pi$ with the corresponding coefficient $c_{v,w}(\cdot):=(\pi(\cdot)v\mid w)\in L^2(G)\setminus\{0\}$.
Since $Z$ is the centre of $G$, it follows that $\pi(Z)$ is contained in the centre of the von Neumann algebra $\pi(G)''$, 
while that von Neumann algebra is a factor, 
hence there exists a continuous homomorphism $\chi\colon G\to\TT$ with $\pi(z)=\chi(z)\id_{\Hc_\pi}$ for all $z\in Z$. 
This implies $\vert c_{v,w}(zg)\vert=\vert c_{v,w}(g)\vert$ 
for all $g\in G$ and $z\in Z$, 
hence, since $c_{v,w}(\cdot):=(\pi(\cdot)v\mid w)\in L^2(G)\setminus\{0\}$, 
we obtain that $Z$ is compact by \cite[A.VII.14]{Ne00}.
On the other hand, as the solvable Lie group $G$ is 1-connected, 
it is diffeomorphic to $\RR^{\dim G}$ 
hence $G$ contains no nontrivial compact subgroups. 
(See \cite[Thms. 11.2.14 and 14.3.12((i)--(ii))]{HN12}.) 
It then follows that $Z=\{\1\}$. 
Then by Lemma~\ref{open_lemma}, for $\xi \in \Oc$ and $\xi_0=\xi \vert_{\dg}$,  
the subset $\Galg\xi_0\subseteq\dg^*$ is open
and, if $\Omega=\dg^\perp+\tilde{G} \xi$, the subset $Gp\subseteq\Bun(\Omega)$ is dense
for every $p\in\Bun(\Omega)$.
Since the bundle projection $\tau\vert_{\Bun(\Omega)}\colon\Bun(\Omega)\to\Omega$ 
is continuous, $G$-equivariant, and surjective, 
it then follows that for every $\eta\in\Omega$ we have $\Omega\subseteq\overline{G\eta}$. 
Thus $\Omega\in(\gg^*/G)^\sim$ by Lemma~\ref{qo}, 
hence $\Omega=\Oc$. 
Therefore $\Bun(\Oc) =\Bun(\Omega)$, thus $Gp$ is dense in $\Bun(\Oc)$ for every $p \in \Bun(\Oc)$ by 
Lemma~\ref{open_lemma} \eqref{open_lemma_item2}.

Moreover, since $\Omega=\dg^\perp+\Galg\xi$, we have  $\Omega=(\iota^*)^{-1}(\iota^*(\Omega)) =(\iota^*)^{-1}(\Galg\xi_0)$. 
Then the hypothesis that the subset $\Galg\xi_0\subseteq\dg^*$ is open implies that the subset $\Omega\subseteq\gg^*$ is open. 
We already established that $\Oc=\Omega$, hence $\Oc\subseteq\gg^*$ is an open subset.

``\eqref{open_th_item2}$\implies$\eqref{open_th_item1}'' 
Select any $\xi\in\Oc$ and denote $\xi_0:=\xi\vert_\dg\in\dg^*$ 
and $\Omega:=\dg^\perp+\Galg\xi$. 
Since $\Oc\in(\gg^*/G)^\sim$ and $\Oc\subseteq\gg^*$ is open, 
we obtain $\Oc=\Omega$ and the subset $\Galg\xi_0\subseteq\dg^*$ is open by Lemma~\ref{open_qo}. 
On the other hand, the hypothesis $\Bun(\Oc)=\O$ is equivalent to the fact that the $G$-orbit of every point in $\Bun(\Omega)$ is dense in $\Bun(\Omega)$, by~\eqref{bun_orbits}. 
Thus Lemma~\ref{open_lemma} is applicable, and shows that the quasi-equivalence class $\ell(\O)$ is square integrable. 
This completes the proof. 
\end{proof}

\section{Unimodular solvable Lie groups do not admit square-integrable representations}\label{sect4}

As mentioned at the end of Section~\ref{sect1}, in the present paper, 
we are mainly interested in the solvable Lie groups which are 1-connected; 
this precludes the tori, for instance, whose dual spaces are actually discrete and which have many square-integrable irreducible representations, and are not covered by  Proposition~\ref{clopen} and Corollary~\ref{clopen_cor}.

In the proof of Proposition~\ref{clopen} and the rest of the paper we use notation and results from \cite[\S 2 and \S 3]{BB21b}.
In particular, for a topological space $X$, we denote by $\Cl(X)$ the space of all closed subsets of $X$ endowed with the upper topology, 
that is, the topology for which a sub-base of closed sets consists of $X$ and the sets $\downarrow F:=\{F'\mid F'\in \Cl(X), \ F'\subseteq F\}$,
	for $F\in\Cl(X)$. 
If $H$ is a topological group acting continuously on $X$, then $(X/H)^\approx$ is the subspace of $\Cl(X)$ consisting of closures of orbits of $H$ in $X$, that is, 
$$ (X/H)^\approx= \{\overline{H\cdot x}\mid x\in X\}.$$
We recall 
that a \emph{clopen} subset of a topological space is a subset that is simultaneously closed and open.

\begin{proposition}\label{clopen}
If $G$ is a $1$-connected solvable Lie group, then $\Prim(G)$ has no clopen singleton subsets.
\end{proposition}

\begin{proof}
Towards a contradiction, let us assume that there exists a clopen singleton subset $F_0:=\{\Pc_0\}\subseteq\Prim(G)$.
We denote $A:=G/[G,G]$; it is a 1-connected abelian Lie group, hence its dual space $\widehat{A}$ is simply connected as well, and
there is a continuous action $\widehat{A}\times\Prim(G)\to \Prim(G)$, $(\chi,\Pc)\mapsto \chi\cdot\Pc$.
We now recall that there exists a  topological embedding $(\Prim(G)/\widehat{A})^\approx\subseteq\Cl(\Prim(G))$, 
and that the  mapping
$q\colon \Prim(G)\to (\Prim(G)/\widehat{A})^\approx$,
$\Pc\mapsto\overline{\widehat{A}\cdot\Pc}$ is open.
(See \cite[Lemma 2.3]{BB21b}.)
Since the group $\widehat{A}$ is connected and generated by any open neighbourhood of $1\in \widehat{A}$,  
and since the set $\{\Pc_0\}\subseteq\Prim(G)$ is open,
it follows that $\widehat{A}\cdot\Pc_0 \subseteq \{\Pc_0\}$. 
The set $\{\Pc_0\}$ is closed, therefore we get that
\begin{equation}\label{clopen_proof_eq1}
\overline{\widehat{A}\cdot\Pc_0}=\{\Pc_0\}=F_0,
\end{equation}
which implies
$F_0\in(\Prim(G)/\widehat{A})^\approx$. 
The set $F_0=\{P_0\}\subseteq\Prim(G)$ is also open, and since  the mapping $q$ is  open, it follows that 
 $q (F_0)\subseteq(\Prim(G)/\widehat{A})^\approx$
is an open subset.
We have  $q(F_0)= q(\{\Pc_0\})=\{F_0\}$ by \eqref{clopen_proof_eq1},
hence
\begin{equation}\label{clopen_proof_eq3}
\{F_0\}\subseteq(\Prim(G)/\widehat{A})^\approx\text{ is an open subset}.
\end{equation}
By the way the topology on $\Cl(\Prim(G))$ was defined  (cf. \cite[Def. 2.1]{BB21b}), we directly obtain that the closure of the singleton set $\{F_0\}$ in $\Cl(\Prim(G))$ is
$\downarrow\{F_0\}:=\{F\in\Cl(\Prim(G))\mid F\subseteq F_0\} = \{\emptyset,F_0\}$.
Therefore the relative closure of~$\{F_0\}$ in $(\Prim(G)/\widehat{A})^\approx$
is
$$\downarrow\{F_0\}\cap(\Prim(G)/\widehat{A})^\approx
=\{F_0\}.$$
It then follows by \eqref{clopen_proof_eq3} that the singleton
$\{F_0\}\subseteq(\Prim(G)/\widehat{A})^\approx$ is a clopen subset.

On the other hand, the solvable Lie group $G$ satisfies the hypothesis of \cite[Thm. 3.6]{BB21b} (cf. the proof of \cite[Prop. 4.7]{BB21c}),
hence there exists a homeomorphism
$$(\Prim(D)/G)^\approx\simeq (\Prim(G)/\widehat{A})^\approx. $$
We thus obtain a nonempty clopen subset of $(\Prim(D)/G)^\approx$.
The group $D=[G, G]$ is a nilpotent Lie group, hence the Kirillov homeomorphism
$\widehat{D}\simeq\dg^*/D$ shows that the topological space $\widehat{D}$ is connected.
Moreover, the  map $\Prim(D)\to (\Prim(D)/G)^\approx$, $\Pc\mapsto \overline{G\cdot \Pc}$ (cf. \cite[Lemma 2.3]{BB21b}) is  continuous and  surjective, hence the topological space $(\Prim(D)/G)^\approx$ is connected as well,
and then it has no nonempty clopen subset.
This contradiction completes the proof.
\end{proof}

\begin{corollary}\label{clopen_cor}
The unimodular $1$-connected solvable Lie groups do not admit square integrable  representations.
\end{corollary}

\begin{proof}
We argue by contradiction. 
Let $G$ be a unimodular 1-connected solvable Lie group and $\pi\colon G\to\Bc(\Hc)$ be a square-integrable (factor) representation,
extended to a $*$-representation $\pi\colon C^*(G)\to\Bc(\Hc)$.
If we denote $\Pc:=\Ker\pi\in\Prim(G)$, then the singleton $\{\Pc\}\subseteq\Prim(G)$ is an open subset since $\pi$ is a square integrable representation, cf. \cite[Cor. 1]{Gr80}.
On the other hand, $\{\Pc\}\subseteq\Prim(G)$ is a closed subset by \cite[Cor. 3.9]{Ros78}, since $G$ is a unimodular connected solvable Lie group.
Thus $\{\Pc\}\subseteq\Prim(G)$ is a clopen subset, which is a contradiction with Proposition~\ref{clopen}.
\end{proof}

Versions of Corollary~\ref{clopen_cor} for the special case of exponential solvable Lie groups have long been known or at least expected; see for instance \cite[Ex. 5.3.8(2.)]{DuRa76} for irreducible representations and \cite[Conj. 3.7]{CuOu12} for monomial representations. 
We also mention 
the interesting fact that the regular 
representation of a unimodular locally compact group has 
admissible vectors if and only if that group is discrete, cf. \cite{Fu02} and \cite[Thm. 2.42]{Fu05}.

\section{Open orbits,  type I representations, and compact open subsets of the primitive ideal space}\label{sect5}

In this section we investigate the relation between open coadjoint orbits of a 1-connected solvable Lie group $G$ on the one hand and type~$\I$ representations on the other hand, and show that the existence of an open coadjoint orbit implies the existence of a compact open subset of $\Prim(G)$.

\subsection{Open quasi-orbits and type~$\I$ representations}\label{open_subsect1}
We start by proving a technical result that 
was noted without proof in \cite[end of \S III.1(4), page 825]{Pu86} 
and also implicitly used in the proof of \cite[\S 4.13, Prop.]{Pu99}. 
We choose to indicate the corresponding construction 
since it actually requires some nontrivial results.

\begin{lemma}
	\label{surj}
	For every $\Oc\in(\gg^*/G)^\sim$ and $\O\in(\Bun(\Oc)/G)^\approx$ we have $\tau(\O)=\Oc$. 
\end{lemma}

\begin{proof}
	We recall from \eqref{alpha} that there is a homeomorphism $\alpha\colon\Bun(\Oc)\to\Oc\times\TT^k$ with $\pr_1\circ\alpha=\tau\vert_{\Bun(\Oc)}$, 
	where $k:=\rk(\Oc)$. 
	Since the torus $\TT^k$ is compact, it follows by \cite[Ch. I, \S 10, no. 3. Cor. 5]{Bo07b}  that $\pr_1$ is a proper map, hence $\tau\vert_{\Bun(\Oc)}\colon\Bun(\Oc)\to\Oc$ is a proper map, too. 
	Then $\tau\vert_{\Bun(\Oc)}\colon\Bun(\Oc)\to\Oc$ is a closed map by \cite[Ch. I, \S 10, no. 1. Prop. 1]{Bo07b}, hence for every subset $A\subseteq \Bun(\Oc)$ we have $\tau(\overline{A})=\overline{\tau(A)}\cap\Oc$,  
	where $\overline{A}$ denotes the closure of $A$ in $\Bun(\Oc)$, while $\overline{\tau(A)}$ is the closure of $\tau(A)$ in $\gg^*$. 
	In particular, for any $p=(\xi,\chi)\in\O$ 
	we obtain 
	$$\tau(\O)=\tau(\overline{Gp})=\overline{\tau(Gp)}\cap\Oc=\overline{G\tau(p)}\cap\Oc
	=\overline{G\xi}\cap\Oc=\Oc$$
	and this completes the proof. 
\end{proof}

\begin{proposition}
	\label{trans_P}
	For $\Oc\in(\gg^*/G)^\sim$ the following assertions are equivalent: 
	\begin{enumerate}[{\rm(i)}]
		\item\label{trans_P_item2}
		 $\Oc$ is a coadjoint orbit, that is, $\Oc \in\gg^*/G$. 
		\item\label{trans_P_item3} Every $\O \in(\Bun(\Oc)/G)^\approx$ is closed, 
		 that is, $(\Bun(\Oc)/G)^\approx=\Bun(\Oc)/G$. 
		 \item\label{trans_P_item4} There is $\O \in(\Bun(\Oc)/G)^\approx$ such that $\O$ is closed. 
		 \end{enumerate}
\end{proposition}

\begin{proof}
	``\eqref{trans_P_item2}$\Rightarrow$\eqref{trans_P_item3}'' 
We must prove that for arbitrarily fixed $p_0=(\xi_0,\chi_0)\in\Bun(\Oc)$, hence $\xi_0\in\Oc$ and $\chi_0\in\mathop{G}\limits^{\mathsout{\wedge}}(\xi_0)$,   
the $G$-orbit 
$$Gp_0=\{(g\xi_0,g\chi_0)\in\Bun(\Oc)\mid g\in G\}$$ 
is a closed subset of $\Bun(\Oc)$. 
It suffices to show that there exists a continuous $G$-equivariant cross-section $\sigma\colon \Oc\to\Bun(\Oc)$ through $p_0$ of the bundle projection 
$\tau\vert_{\Bun(\Oc)}\colon \Bun(\Oc)\to\Oc$, 
that is, the continuous mapping $\sigma$ satisfies the conditions 
$$
 \begin{cases}
 \sigma(g\xi)=g \sigma(\xi)  &  \text{for all  } \xi\in\Oc, \, g\in G,\\
  \tau(\sigma(\xi))=\xi &  \text{for all  } \xi\in\Oc, \, g\in G, \\
\sigma(\xi_0)=p_0.& 
\end{cases}
$$
Indeed, if such a continuous mapping $\sigma$ exists, then its image $\sigma(\Oc)$ is a closed subset of $\Bun(\Oc)$ (since $\tau\circ\sigma=\id_\Oc$) and on the other hand $\sigma(\Oc)=Gp_0$ 
(since $\Oc=G\xi_0$ and $\sigma$ is $G$-equivariant), and thus the $G$-orbit $g.p_0$ is closed in $\Bun(\Oc)$.

We show that the mapping 
\begin{equation}
\label{trans_proof_eq1}
\sigma\colon \Oc\to\Bun(\Oc),\quad \sigma(g\xi_0):=(g\xi_0,g\chi_0)
\end{equation}
is well defined and continuous. 
It will then directly follow that $\sigma$ is $G$-equivariant and satisfies $\sigma(\xi_0)=(\xi_0,\chi_0)=p_0$ and $\tau\circ\sigma=\id_\Oc$. 
Let us also note that  $\sigma$ is defined throughout $\Oc$ since $\Oc=G\xi_0$. 

In order to show that  $\sigma$ in \eqref{trans_proof_eq1} is well defined, we must check that if $g_1,g_2\in G$ and $g_1\xi_0=g_2\xi_0$, then $\sigma(g_1\xi_0)=\sigma(g_2\xi_0)$. 
To this end  it is enough to prove that if $g\in G$ and $g\xi_0=\xi_0$ then $g\chi_0=\chi_0$. 
In fact, the condition $g\xi_0=\xi_0$ means $g\in G(\xi_0)$, 
thus $g\chi_0=\chi_0$ by \eqref{invar}.
This completes the proof of the fact that  $\sigma$ is well defined.

To prove that  $\sigma\colon\Oc\to\Bun(\Oc)$ is smooth
we have to show  that for every admissible $b\colon \Oc \to K$ and for 
$$f_b\colon\Bun(\Oc)\to\TT, \quad f_b(\xi,\chi)=\chi(b(\xi)),$$ 
the corresponding mapping $f_b\circ\sigma\colon\Oc\to\TT $ is smooth. 
The orbit map $\alpha^{\xi_0}\colon G\to\Oc$, $\alpha^{\xi_0}(g):=g\xi_0$ is a submersion since $\Oc$ is locally compact. 
Therefore, 
it is enough to check that  $f_b\circ\sigma\circ\alpha^{\xi_0}\colon G\to\TT$ is smooth. 
For every $g\in G$ we have, by \eqref{trans_proof_eq1}, 
\begin{align*}
(f_b\circ\sigma\circ\alpha^{\xi_0})(g)
&=f_b(\sigma(g\xi_0))=f_b(g\xi_0,g\chi_0)=(g\chi_0)(b(g\xi_0))=\chi_0(g^{-1}b(g\xi_0)g) \\
&=\chi_0(g^{-1}b(\alpha^{\xi_0}(g))g)  \\
&=\varphi_{\chi_0,\xi_0}(g^{-1}b(\alpha^{\xi_0}(g))g)
\end{align*}
where $\varphi_{\chi_0,\xi_0}\in\Hom(K,\TT)$ satisfies $$\text{$\varphi_{\chi_0,\xi_0}\vert_{\overline{G}(\xi_0)}=\chi_0\in\Hom(\overline{G}(\xi_0),\TT)$ and $\varphi_{\chi_0,\xi_0}\vert_D=\ee^{\ie\xi_0\circ\log_D}\in\Hom(D,\TT)$, }
$$
cf. \cite[page 494]{Pu71}. 
The above equality shows that $f_b\circ\sigma\circ\alpha^{\xi_0}\colon G\to\TT$ is smooth, and this completes the proof of the implication. 

	The implication ``\eqref{trans_P_item3}$\Rightarrow$\eqref{trans_P_item4}''  is trivial.
	
	``\eqref{trans_P_item4}$\Rightarrow$\eqref{trans_P_item2}'' 
	This follows by Lemma~\ref{surj} since the mapping $\tau\colon\Bun(\gg^*)\to\gg^*$ is $G$-equivariant. 
\end{proof}

\begin{remark}\label{new_typeI}
\normalfont
Using Proposition~\ref{trans_P}, we can rewrite \eqref{open_cor1_proof_eq1}
as follows: if $\Oc \in (\gg^*/G)^\sim$ and  $\O\in(\Bun(\Oc)/G)^\approx$, then
\begin{equation}\label{2.1'}
\ell(\O)\text{ is type~$\I$} \iff 
	\Oc \in \gg^*/G \text{ and }
	[G(\xi):\overline{G}(\xi)]<\infty\text{ for some }\xi\in\Oc.
	\end{equation}
	In particular, since the condition in the right hand side of \eqref{2.1'} depends on $\Oc$ only, we get that if there is $\O_0 \in (\Bun(\Oc)/G)^\approx$ such that $\ell (\O_0)$ is type~$\I$, then $\ell(\O)$ is type~$\I$ for every $\O \in (\Bun(\Oc)/G)^\approx$. 
\end{remark}

\begin{theorem}\label{open_cor}
Let $G$ be a $1$-connected solvable Lie group with its Lie algebra $\gg$. 
Assume that there is a coadjoint orbit $\Oc\in\gg^*/G$ such that $\Oc$ is an open subset of $\gg^*$.
Then for every $\O\in (\Bun(\Oc)/G)^\approx$, 
its corresponding  quasi-equivalence class $\ell(\O)\in \stackrel{\frown}{G}_\nor$ is  type~$\I$. 
\end{theorem}

\begin{proof}
	Since $\Oc\in\gg^*/G$ and $\Oc\subseteq\gg^*$ is an open subset, 
	we have $\gg(\xi)=\{0\}$, that is, $G(\xi)_\1=\{\1\}$, 
for arbitrary $\xi\in\Oc$.   
Then, by \eqref{fin-gen}, $G(\xi)$ is a finitely generated free abelian group. 
It then follows by the definition of $\overline{G}(\xi)$ that $\overline{G}(\xi)=G(\xi)$. 
By hypothesis,  $\Oc\in\gg^*/G$, hence   
$\ell(\O)$ is type~$\I$ by \eqref{2.1'}.
\end{proof}

\begin{theorem}
\label{F_gen}
Let $G$ be a $1$-connected solvable Lie group with its Lie algebra $\gg$. 
Assume that there is a coadjoint orbit $\Oc\in\gg^*/G$ such that 
$\Oc$ is an open subset of~$\gg^*$.
The following assertions are equivalent: 
\begin{enumerate}[{\rm(i)}]
	\item\label{F_gen_item1} 
	$G(\xi)=\{\1\}$ for all $\xi\in\Oc$.
	\item\label{F_gen_item4} 
	The quasi-equivalence class $\ell(\O)\in \stackrel{\frown}{G}_\nor$ 
	is square integrable. 
	\item\label{F_gen_item2} 
	The mapping $\tau\vert_{\Bun(\Oc)}\colon \Bun(\Oc)\to\Oc$ is injective. 
	\item\label{F_gen_item3} 
    $(\Bun(\Oc)/G)^\approx=\{\O\}$.
 \end{enumerate}
If these conditions are satisfied, then the centre of $G$ is trivial. 
\end{theorem}

\begin{proof}
	``\eqref{F_gen_item1}$\Leftrightarrow$\eqref{F_gen_item2}'' 
For every $\xi\in\Oc$, its corresponding fibre $\tau^{-1}(\xi)$ is homeomorphic to
the torus which is the dual group of the finitely-generated free abelian group $\overline{G}(\xi)/G(\xi)_\1$, 
cf. \cite[Ch. II, \S 6.3.a, page 537]{Pu71}. 
The proof of Theorem~\ref{open_cor} shows that $\overline{G}(\xi)/G(\xi)_\1=G(\xi)$,
hence $G(\xi)=\{\1\}$ if and only if the fibre $\tau^{-1}(\xi)$ consists of only
one point. 

The implication ``\eqref{F_gen_item2}$\Rightarrow$\eqref{F_gen_item3}'' is clear.

``\eqref{F_gen_item1}$\Rightarrow$\eqref{F_gen_item4}'' 
If  $G(\xi)=\{\1\}$ for some $\xi\in\Oc$, then 
$\rk(\xi)=0$.
 This shows that the mapping 
$\tau\vert_{\Bun(\Oc)}\colon\Bun(\Oc)\to\Oc$ is bijective. 
Since $\tau$ is $G$-equivariant, it then follows by Proposition~\ref{trans_P} that $\O=\Bun(\Oc)$. 
The group $G$ has trivial centre, since it is contained in $G(\xi)=\{1\}$.
Therefore $[\pi]^\frown\in \stackrel{\frown}{G}_\nor$ is square integrable by Theorem~\ref{open_th}.

``\eqref{F_gen_item3}$\Rightarrow$\eqref{F_gen_item2}'' 
Since $\Oc\in\gg^*/G$, it follows by the proof of Proposition~\ref{trans_P} 
that there exists a continuous $G$-equivariant mapping $\sigma\colon\Oc\to\Bun(\Oc)$ with $\tau\circ\sigma=\id_{\Oc}$. 
This implies that $\sigma(\Oc)\subseteq\Bun(\Oc)$ is a closed $G$-invariant subset, in particular $\sigma(\Oc)\in (\Bun(\Oc)/G)^\approx$. 
But $(\Bun(\Oc)/G)^\approx=\{\Bun(\Oc)\}$, hence $\sigma(\Oc)=\Bun(\Oc)$. 
Since $\sigma$ is a continuous cross-section of the mapping  $\tau\vert_{\Bun(\Omega)}\colon\Bun(\Omega)\to\Omega$, it then follows 
that the fibres of $\tau\vert_{\Bun(\Omega)}$ are singletons.

``\eqref{F_gen_item4}$\Rightarrow$\eqref{F_gen_item3}'' 
This follows from Theorem~\ref{open_th}.

Finally, since the centre of $G$ is contained in  $G(\xi)=\{\1\}$ for any $\xi\in\Oc$, 
if condition \eqref{F_gen_item1} is satisfied then the centre of $G$ is trivial.  
\end{proof}

\begin{corollary}\label{open_cor2}
	If $G$ is a 1-connected solvable Lie group, 
	then $[\pi]^\frown\mapsto \tau(\ell^{-1}([\pi]^\frown))$ gives a one-to-one correspondence between the following sets: 
	\begin{itemize}
		\item the equivalence classes of square-integrable irreducible representations of $G$;
		\item\label{sq} the type~$\I$ square-integrable 
		quasi-equivalence classes $[\pi]^\frown\in \stackrel{\frown}{G}_\nor$; 
		\item\label{orb} the simply connected open coadjoint orbits of $G$. 
	\end{itemize}
\end{corollary}

\begin{proof}
We first note that if $G$ has a nontrivial centre, the above sets are empty.
Therefore we may assume that $G$ has trivial centre. 
We also recall the canonical one-to-one correspondence between the quasi-equivalence classes of type~$\I$ factor representations and the unitary equivalence classes of irreducible representations, cf. \cite[Prop. 5.4.11 and 5.3.3]{Di64}. 

Let $[\pi]^\frown\in \stackrel{\frown}{G}_\nor$, $\Oc\in (\gg^*/G)^\sim$, and  $\O\in (\Bun(\Oc)/G)^\approx$ with  $\ell(\O)=[\pi]^\frown$. 
In particular $\Oc=\tau(\ell^{-1}([\pi]^\frown))$.

If  
$[\pi]^\frown\in \stackrel{\frown}{G}_\nor$ is square integrable and type~$\I$, 
then $\Oc\subseteq\gg^*$ is open by Theorem~\ref{open_th}, 
$\Oc\in\gg^*/G$ and $G(\xi)=\{\1\}$ for any $\xi\in\Oc$ by Theorem~\ref{F_gen}. 
Hence $\Oc$ is an open coadjoint orbit and  the mapping 
$G\to\Oc$, $g\mapsto g\xi$ is a diffeomorphism, so $\Oc$ is simply connected as well. 

Conversely, assume that  $\Oc\in\gg^*/G$  and it is open and simply connected.
Since both $G$ and $\Oc$ are simply connected, $G(\xi)$ is connected for any $\xi \in \Oc$. 
Then $G(\xi)_\1=G(\xi)=\{\1\}$.
By Theorem~\ref{F_gen} (\eqref{F_gen_item1}$\Rightarrow$\eqref{F_gen_item4}) we get that 
$[\pi]^\frown=\ell(\O)$ is square integrable, and type~$I$  
by 
Theorem~\ref{open_cor}.
 hence $\rk(\xi)=0$, and this shows that the mapping 
$\tau\vert_{\Bun(\Oc)}\colon\Bun(\Oc)\to\Oc$ is bijective. 
By Proposition~\ref{trans_P} (\eqref{trans_P_item2}$\Rightarrow$\eqref{trans_P_item3})
we obtain
$\O=\Bun(\Oc)$. 
The group $G$ has trivial centre by assumption, hence $[\pi]^\frown\in \stackrel{\frown}{G}_\nor$ is square integrable by Theorem~\ref{open_th}.

Finally, if either  $\Oc\in\gg^*/G$ is an open simply connected coadjoint orbit, or $[\pi]^\frown\in \stackrel{\frown}{G}_\nor$ is square integrable and type~$\I$, then we have seen above that $G(\xi)=\{\1\}$ and $\tau\vert_{\Bun(\Oc)}\colon\Bun(\Oc)=\O\to\Oc$ is bijective, hence 
$\O$ is uniquely determined by~$\Oc$. 
\end{proof}

\subsection{Open coadjoint orbits and compact open subsets of $\Prim(G)$}\label{open_subsect2}

Let now $\Oc\in\gg^*/G$ be a coadjoint orbit that is an open subset of $\gg^*$. 
Since the $G$-orbits in $\Bun(\Oc)$ are closed, the map
$$ \kk_\Oc\colon \Bun(\Oc)/G\to\Prim(G), \; \O \mapsto J(\O)$$ 
is injective  (since the map \eqref{J} is bijective) and continuous by \eqref{Jc}.

\begin{theorem}
\label{homeo}
Let $G$ be a 1-connected solvable Lie group. 
If a coadjoint orbit $\Oc\in\gg^*/G$ is an open subset of $\gg^*$, 
then the mapping $\kk_\Oc\colon\Bun(\Oc)/G\to\Prim(G)$ 
is open, and a homeomorphism onto its image. 
Moreover, its image is transitively acted on by the character group~$X(G)$. 
\end{theorem}

\begin{proof}
We prove first that $\kk_\Oc$ is  a homeomorphism onto its image.

The  group 
$X(G)\simeq \widehat{G/D}$ 
acts canonically on $\Prim(G)$ by 
\begin{equation}
	\label{homeo_proof_eq0}
	X(G)\times\Prim(G)\to\Prim(G),\quad (\chi,J)\mapsto\chi\cdot J
\end{equation}
cf., e.g., \cite[Thm. 3.6]{BB21b}. 
Recall that $\chi\cdot J:=\Ker(\chi T)\in\Prim(G)$ if $T\colon G\to\Bc(\Hc)$ is a factor unitary representation with $\Ker T=J\in\Prim(G)$, while $\chi T\colon G\to\Bc(\Hc)$, $g\mapsto \chi(g)T(g)$. 
The relative topology of every orbit of the group action~\eqref{homeo_proof_eq0} is Hausdorff by \cite[Lemma 1]{B82}. 

We already know that $\kk_\Oc$ is continuous, injective, and its domain $\Bun(\Oc)/G$ is compact and Hausdorff. 
Therefore, it suffices to show that the relative topology of its image  $\kk_\Oc(\Bun(\Oc)/G)\subseteq\Prim(G)$ is Hausdorff. 
Furthermore, by the remarks above, it is enough to show that $\kk_\Oc(\Bun(\Oc)/G)$ is an $X(G)$-orbit. 

To this end we fix $\xi_0\in\Oc$ and we denote $p_0:=(\xi_0,\1_{G(\xi_0)})\in\Bun(\Oc)$, 
where $\1_{G(\xi_0)}:=\1_G\vert_{G(\xi_0)}$. 
We  show that for every $p\in \Bun(\Oc)$, 
\begin{equation}
\label{homeo_proof_eq1}
\Jc(p)\in X(G)\cdot\Jc(p_0), 
\end{equation}
that is, 
$$
\kk_\Oc(\Bun(\Oc)/G)=X(G)\cdot\Jc(p_0).
$$
Since $G(\xi_0)\cap D=\{\1\}$, the quotient map $G\to G/D$ 
gives an isomorphism of $G(\xi_0)$ onto a certain discrete subgroup of $G/D$. 
Since $G/D$ is isomorphic to a vector group, it then  follows that 
every character of~$G(\xi_0)$ extends to a character of~$G$. 
That is, the restriction map 
$$X(G)\to \widehat{G(\xi_0)}, \quad \chi\mapsto \chi\vert_{G(\xi_0)}$$ 
is surjective.  

Let $\chi_0\in \widehat{G(\xi_0)}$ be arbitrary, and denote $p: =(\xi_0, \chi_0)$.
Pick $\chi\in X(G)$ with 
$\chi_0=\chi\vert_{G(\xi_0)}$ and let $\sigma\in \dg^*$ be such that $\ie \sigma=d\chi$. 
Then, by the definition of the action of $X(G)$ in $\Bun(\gg^*)$, we may write
$$ p =(\xi_0, \chi_0)= \chi(\xi_0-\sigma,\1_{G(\xi_0)} ).$$
Therefore, by \eqref{1002} we have that 
\begin{equation}\label{1003}
\lambda(p)= \chi\vert_{K} \lambda((\xi_0-\sigma, \1_{G(\xi_0)})).
\end{equation}
On the other hand, $\gg(\xi_0)=\{0\}$, hence $g(\xi_0)+\dg= \dg$ and 
$ (\xi_0-\sigma)\vert_{\gg(\xi_0)+\dg} = \xi_0\vert_{\gg(\xi_0)+\dg}$.
Then, by \cite[Lemma~6.2, page 501]{Pu71}, there is $a\in G(\xi_0\vert_\dg)_\1$ such that
$ \xi_0-\sigma = a \xi_0.$. 
Moreover, $a\1_{G(\xi_0)}= \1_{G(\xi_0)}$.
Therefore, \eqref{1003} becomes
$$
\lambda(p)= \chi\vert_{K} \lambda(a p_0) = \chi\vert_{K} \lambda( p_0), 
$$
where we have used \eqref{1001} to obtain the last equality. 
Thus we get that
\begin{equation}\label{1004}
T(p) =\Ind_K^G \lambda(p) =\Ind_K^G (\chi\vert_{K} \lambda( p_0))= \chi \Ind_K^G \lambda(p_0)=\chi T(p_0),
\end{equation}
where the latter equality follows from \cite[Thm. 2.58]{KT13}. 
Taking kernels in 
both sides of \eqref{1004}, we obtain \eqref{homeo_proof_eq1}.

Next we show that  $\kk_\Oc$ is open. 
Recall that $R\colon \Prim(G) \to (\widehat{D}/G)^\approx$ is the continuous surjective mapping from  \cite[Lemma 3.5]{BB21b}, defined by $R(J): =\supp T\vert_{D}$ for every
$J\in \Prim(G)$ and $T$ factor representation such that $J=\Ker T$.
Consider also
the map $Q\colon \widehat{D}\to (\widehat{D}/G)^\approx$, 
$Q(\pi):=\overline{G\pi}$ for every $\pi\in\widehat{D}$. 
(This is actually  $\iota\circ Q$ in \cite[Lemma 2.3]{BB21b}.)
Let $q_D\colon \dg^* \to \dg^*/D$ be the canonical quotient map, and denote
$\nu\colon \gg^*\to \widehat{D}$, $\nu := \kappa_D\circ q_D\circ\iota^*$.
Then, if  $\varepsilon \colon \Bun(\gg^*) \to \widehat{D}$,  $\varepsilon (p) = \nu(\tau (p))$,  
 the diagram 
\begin{equation}\label{diag}
	\xymatrix{
\dg^* \ar[d]_{q_D} & \gg^* \ar[l]_{\iota^*} \ar[dr]_{\nu}& \Bun(\gg^*) \ar[l]_{\tau} \ar[r]^{\Jc} \ar[d]^{\varepsilon} & \Prim(G) \ar[d]^{R} \\
	\dg^*/D \ar[rr]_{\kappa_D} & & \widehat{D} \ar[r]^{Q} & (\widehat{D}/G)^\approx
}
\end{equation}
is commutative
  since we have 
\begin{equation}
	\label{restr}
	R(\Jc(p))=\overline{G\varepsilon(p)}\; \; \text{for all }  p\in\Bun(\gg^*)), 
\end{equation}
as noted in \cite[III.3.d.3]{Pu73}.
To show that $\kk_\Oc$ is open, 
 it is enough to show that
\begin{equation}\label{claim2}
 \Jc(\Bun(\Oc)) =R^{-1}((Q\circ\nu)(\Oc)).
\end{equation}
Indeed, 
$R$ is continuous, the subset $\nu(\Oc) = (\kappa_D\circ q_D\circ\iota^*)(\Oc)=G\pi_0\subseteq\widehat{D}$ is open, and 
	the mapping $Q$ is open by \cite[Lemma 2.3]{BB21b}, 
	therefore equality~\eqref{claim2} shows that
$\Jc(\Bun(\Oc))=\kk_\Oc(\Bun(\Oc)/G)$ is an open subset of $\Prim(G)$. 
Using that $\kk_\Oc$ is a homeomorphism onto its image we get that $\kk_\Oc$ is open.

We prove now equality \eqref{claim2}.
\\
	``$\subseteq$'': 
The above diagram \eqref{diag} is commutative, hence
	$$R\circ\Jc=Q\circ\varepsilon=Q\circ \nu \circ\tau.$$
Then, since $\tau(\Bun(\Oc))=\Oc$, we obtain 
$$R(\Jc(\Bun(\Oc)))
=(Q\circ \nu \circ\tau)(\Bun(\Oc))
	=(Q\circ\nu)(\Oc)
$$	
and this proves the inclusion $\subseteq$. 
	
	``$\supseteq$'': 
	We must prove that for arbitrary $J\in\Prim(G)$ with
	\begin{equation}
		\label{res_prop_proof_eq1}
		R(J)\in(Q\circ\nu )(\Oc)
	\end{equation}
	there exists $p\in\Bun(\Oc)$ with  $J=\Jc(p)$. 
	Since the mapping $\Jc$ is surjective, we already know that there exists $p=(\xi,\chi)\in\Bun(\gg^*)$ with $\Jc(p)=J$.  
	We claim that 
	\begin{equation}\label{claim}
\xi=\tau(p)\in\Oc, 
	\end{equation}
	that is, $p\in\Bun(\Oc)$.
	
	We prove \eqref{claim}.
By \eqref{res_prop_proof_eq1} there is $\xi_0\in\Oc$ such that 
$$R(J)=(Q\circ\nu )(\xi_0) = \overline{G\pi_0}, $$
where
$$\pi_0:=\nu (\xi_0)=\kappa_D(D\xi_0\vert_\dg)\in\widehat{D}.$$ 
On the other hand, \eqref{restr} shows that
$$R(J)=R(\Jc(p))=\overline{G\pi}$$
where $\pi:=\varepsilon(p)=\kappa_D(D\xi\vert_D)\in\widehat{D}$. 
		Therefore $\overline{G\pi}=\overline{G\pi_0}$; in particular $\pi_0\in\overline{G\pi}$. 
	
	In addition, $\Oc=G\xi_0$, and since 
	the mappings $\kappa_D,q_D,\iota^*$ are $G$-equivariant, we have
	$$
	\nu (\Oc)=(\kappa_D\circ q_D\circ\iota^*)(G\xi_0)=
	G(\kappa_D\circ q_D\circ\iota^*)(\xi_0)= G \nu (\xi_0)= G\pi_0.
	$$
	Then, as $\Oc\subseteq\gg^*$ is an open subset and $\kappa_D,q_D,\iota^*$ are open mappings, it follows that the orbit $G\pi_0$ is an open neighbourhood of $\pi_0$. 
	Hence the relation $\pi_0\in\overline{G\pi}$ implies $G\pi_0\cap G\pi\ne\emptyset$, thus
	$G\pi_0=G\pi$. 
	Consequently, there exists $g\in G$ with $\pi=g\pi_0$. 
	Since the Kirillov homeomorphism $\kappa_D$ is $G$-equivariant, 
	it then follows by the definition of $\pi_0$ and $\pi$ that $\xi\vert_\dg\in Dg\xi_0\vert_\dg$. 
	This implies $\xi\in Dg\xi_0+\dg^\perp\subseteq G\xi_0+\dg^\perp$. 
	Since $G\xi_0\subseteq\gg^*$ is open, we have $G\xi_0=G\xi_0+\dg^\perp$ by 
	\cite[Lemma 3.4]{BB21b}, 
	hence $\xi\in G\xi_0=\Oc$. 
	This completes the proof of \eqref{claim} and of \eqref{claim2}, and of the theorem. 
	\end{proof}

\begin{corollary}\label{compop}
If $\Oc\in\gg^*/G$ and $\Oc$ is an open subset of $\gg^*$, then 
$\kk_\Oc(\Bun(\Oc)/G)$ is a compact open subset of $\Prim(G)$.
\end{corollary}

\begin{proof}
We consider the following space of continuous mappings 
$$\Gamma_G(\Oc):=\{\sigma\in\Cc(\Oc,\Bun(\Oc))\mid\tau\circ\sigma=\id_\Oc\text{ and }\sigma\text{ is $G$-equivariant}\}$$
endowed with the topology of pointwise convergence. 
For arbitrary $\xi_)\in\Oc$ we define the evaluation map $$\ev_{\xi_0}\colon\Gamma_G(\Oc)\to\tau^{-1}(\xi_0),\quad \ev_{\xi_0}(\sigma):=\sigma(\xi_0)$$ 
and the range map 
$$\Ran_\Oc\colon\Gamma_G(\Oc)\to\Bun(\Oc)/G,\quad \Ran_\Oc(\sigma):=\sigma(\Oc).$$
As a by-product of the proof of Proposition~\ref{trans_P}, the mappings $\ev_\Oc$ and $\Ran_\Oc$ are well-defined homeomorphisms, hence we obtain the homeomorphism 
$$\Ran_\Oc\circ(\ev_{\xi_0})^{-1}\colon \tau^{-1}(\xi_0)\to\Bun(\Oc)/G.$$
Here $\tau^{-1}(\xi_0)=\{\xi_0\}\times\widehat{G(\xi_0)}$, 
hence $\Bun(\Oc)/G$ is homeomorphic to the torus $\widehat{G(\xi_0)}$. 
In particular, in Theorem~\ref{homeo}, the open subset $\kk_\Oc(\Bun(\Oc)/G)\subseteq\Prim(G)$ is also the continuous image through $\kk_\Oc$ 
of the torus $\Bun(\Oc)/G$, hence $\kk_\Oc(\Bun(\Oc)/G)$ is a compact open subset of $\Prim(G)$. 
\end{proof}

\section{Solvable Lie groups with nilradicals of  codimension $1$}\label{sect6} 
In this section we use Theorem~\ref{open_th} to  show that every square integrable representation of a solvable Lie group is type~$\I$ if the nilradical has codimension~1 (Theorem~\ref{codim1_th}).  
We achieve that result via the description of the type-$\I$ property in terms of the coadjoint action (Theorem~\ref{F_gen}, \ref{open_cor2}). 

Our next application of Theorem~\ref{open_th} is Theorem~\ref{codim1_th} and gives an improvement of \cite[Thm. 4.5]{KT96}. 
To this end we recall the following facts. 

\begin{remark}\label{isolated}
	\normalfont 
	The mapping 
	$$\ker\colon \stackrel{\frown}{G}_\nor\to\Prim(G),$$
	which takes every quasi-equivalence class in $\stackrel{\frown}{G}_\nor$ 
	to the kernel in $C^*(G)$ of any representation of  that class,  
	is bijective by \cite[Thm. 1, page 119]{Pu74}. 
	It is known from \cite{Gr80} that the mapping $\ker$ 
	gives a one-to-one correspondence between  
	\begin{itemize}
		\item the square-integrable 
		classes $[\pi]^\frown\in \stackrel{\frown}{G}_\nor$; 
		\item the isolated points of $\Prim(G)$. 
	\end{itemize} 
	A point $\Pc\in\Prim(G)$ is isolated if and only if 
	the subset $\{\Pc\}\subseteq\Prim(G)$ is open. 
	
	A primitive ideal $\Pc\in\Prim(G)$ is said to be \emph{type~$\I$} if 
	its corresponding $[\pi]^\frown\in \stackrel{\frown}{G}_\nor$ with $\ker([\pi]^\frown)=\Pc$ is type~$\I$ or, equivalently, if there exists 
	a normal irreducible representation $\pi$ of $G$ with $\ker([\pi]^\frown)=\Pc$. 
	(See \cite[\S 3, Lemma~3.1]{Pu74}.)
\end{remark}

\begin{theorem}\label{codim1_th}
If $G$ is a $1$-connected solvable Lie group with its nilradical $N$ and $\dim(G/N)=1$, then the isolated points of  the primitive ideal spectrum $\Prim(G)$ are exactly the kernels in $C^*(G)$ of the square-integrable irreducible representations of $G$. 
 \end{theorem}

To prove Theorem~\ref{codim1_th} we are going to show that, in the conditions of the theorem, every open 
coadjoint quasi-orbit is actually an orbit, and then use Theorem~\ref{open_cor}.
This requires several lemmas, in addition to Theorem~\ref{open_th}. 

\begin{lemma}\label{codim1_lemma1}
	Let $\ng$ be a nilpotent Lie algebra with its corresponding nilpotent Lie group $N=(\ng,\cdot)$. 
	If $D\in\Der(\ng)$ and $\alpha\colon\RR\to\Aut(N)$, $\alpha_t:=\ee^{tD}$, then the coadjoint action of the semidirect product $G:=N\rtimes_{\alpha}\RR$ is given by 
	$$\Ad_G^*((x,t)^{-1})\colon \ng^*\dotplus\RR\to\ng^*\dotplus\RR,
	\quad \Ad_G^*((x,t)^{-1})=
	\begin{pmatrix}
	\ee^{tD^*}\Ad_N^*(-x) & 0 \\
	\psi(\ad_\ng x)Dx & 1
	\end{pmatrix}$$
	for arbitrary $x\in\ng$ and $t\in\RR$, 
	where we regard $\psi(\ad_\ng x)Dx\in\ng$ as a linear functional on $\ng^*$ via the canonical isomorphism $\ng\simeq(\ng^*)^*$, and 
	$\psi(z):=-\sum\limits_{k\ge 0}\frac{1}{(k+1)!}z^k$ for all $z\in\CC$. 
\end{lemma}

\begin{proof}
	For all $t,s\in\RR$ and $x,y\in\ng$ we have 
	$(x,t)\cdot(y,s)=(x\cdot\alpha_t(y),t+s)$ and $(x,t)^{-1}=(\alpha_{-t}(-x),-t)$ in $G=N\rtimes_{\alpha}\RR$, hence 
	\begin{align*}
	(x,t)\cdot(y,s)\cdot(x,t)^{-1}
	&=(x\cdot\alpha_t(y),t+s)\cdot(\alpha_{-t}(-x),-t) \\
	&=(x\cdot\alpha_t(y)\cdot\alpha_{t+s}(\alpha_{-t}(-x)),s)\\
	&=(x\cdot\alpha_t(y)\cdot\alpha_s(-x),s). 
	\end{align*}
	Thus
	$$
(x,t)\cdot(y,s)\cdot(x,t)^{-1}=\begin{cases}
	(x\cdot\alpha_t(y)\cdot(-x),0)&\text{ if }s=0,\\
	(x\cdot\alpha_s(-x),s)&\text{ if }y=0.
	\end{cases}
	$$
	
	By differentiation at $(y,s)=(0,0)\in G$ we then obtain 
	$$(\Ad_G(x,t))(y,0)=\bigl(((\Ad_N x)\circ\ee^{tD})(y),0\bigr)$$
	and 
	\begin{align*}
	(\Ad_G(x,t))(0,1)
	&=\Bigl(\frac{\de}{\de s}\Bigl\vert_{s=0}(x\cdot\ee^{sD}(-x)),1\Bigr)
	\end{align*}
	On the other hand, by \cite[Ch. II, \S 6, no. 5, Prop. 5]{Bo06}
	$$(u+v)\cdot (-u)=\sum_{k\ge 0}\frac{1}{(k+1)!}(\ad_\ng u)^kv+ O(v^2),$$
	hence, replacing $v$ by $-v$ and then multiplying both sides by $-1$, 
	$$u\cdot(-u+v)=\sum_{k\ge 0}\frac{1}{(k+1)!}(\ad_\ng u)^kv+ O(v^2).$$
	For $u=x$ and $v=x-\ee^{sD}x=O(s)$ we then obtain  
	$$x\cdot\ee^{sD}(-x)=\sum_{k\ge 0}\frac{1}{(k+1)!}(\ad_\ng x)^k(x-\ee^{sD}x)+ O(s^2)$$
	hence 
	$$\frac{\de}{\de s}\Bigl\vert_{s=0}(x\cdot\ee^{sD}(-x))
	=-\sum_{k\ge 0}\frac{1}{(k+1)!}(\ad_\ng x)^kDx. 
	$$
	Consequently 
	$$\Ad_G(x,t)\colon\ng\dotplus\RR\to\ng\dotplus\RR,\quad 
	\Ad_G(x,t)=\begin{pmatrix}
	\Ad_N(x)\ee^{tD} & \psi(\ad_\ng x)Dx \\
	0 & 1
	\end{pmatrix}$$
	and, then 
	$$\Ad_G^*((x,t)^{-1})
	=\Ad_G(x,t)^*
	=\begin{pmatrix}
	\ee^{tD^*}\Ad_N(x)^* & 0 \\
	\psi(\ad_\ng x)Dx & 1
	\end{pmatrix}
	=\begin{pmatrix}
	\ee^{tD^*}\Ad_N^*(-x) & 0 \\
	\psi(\ad_\ng x)Dx & 1
	\end{pmatrix}$$
	and this completes the proof. 
\end{proof}

\begin{lemma}\label{loc-closed}
Let $\Vc$ be a finite-dimensional real vector space. 
For every $A\in\End(\Vc)$ with $\sigma(A)\cap\ie\RR=\emptyset$ 
and every $v\in\Vc\setminus\{0\}$ the mapping 
$\RR\to\Vc$, $t\mapsto \ee^{tA}v$, 
is a homeomorphism onto its image, and its image is a locally closed subset of~$\Vc$.  
\end{lemma}

\begin{proof}
See \cite[Lemma 5.3]{BB21a} and \cite[Proof of Prop. 2.14]{BB18b}.
\end{proof}

We now make the following remark for later use in the proof of Lemma~\ref{codim1_lemma2}. 

\begin{remark}
\normalfont
	For every connected solvable Lie group $G$ with its Lie algebra $\gg$ and
every ideal $\hg\subseteq\gg$ we have 
\begin{equation}\label{quasi_eq2}
r(\xi)\subseteq\xi+\hg^\perp \quad \text{for every $\xi\in \gg^*$ such that  $\hg\subseteq\gg(\xi)$}.
\end{equation}
Indeed, since  $[\gg,\hg]\subseteq\hg$ and the Lie group $G$ is connected, 
we obtain $G\hg\subseteq\hg$, hence $G\hg^\perp\subseteq\hg^\perp$. 
Then fix any symmetric open neighbourhoods $V_0$ of $0\in\gg$ and $U_1$ of $\1\in G$  with the property that 
$\exp_G\vert_{V_0}\colon V_0\to U_1$ is a diffeomorphism and 
for all $x\in V_0$ 
the series $\sum\limits_{k\ge 0}\frac{1}{k!}(\ad_\gg x)^k$ is convergent in $\Aut(\gg)$. 
Then for all $g=\exp_G x\in U_1$ with $x\in V_0$ and all 
$y\in\hg$ we have 
$$\langle g^{-1}\xi,y\rangle=\langle \xi,gy\rangle 
=\langle \xi,\sum\limits_{k\ge 0}\frac{1}{k!}(\ad_\gg x)^ky\rangle 
\in\langle\xi,y\rangle +\langle \xi, [x,\hg]\rangle
=\{\langle\xi,y\rangle\} $$
since $\hg\subseteq\gg(\xi)$. 
Thus, for every $g\in U_1$,  $g\xi=\xi$ on  $\hg^\perp$. 
Moreover, since $G$ is connected, for arbitrary $g\in G$ there exist $g_1,\dots,g_m\in U_1$ with $g=g_1\cdots g_m$, 
and we have 
\begin{align*}
g\xi-\xi
& =g_1\cdots g_{m-1}(g_m\xi-\xi)
+\cdots +g_1(g_2\xi-\xi)+(g_1\xi-\xi) \\
&\in g_1\cdots g_{m-1}\hg^\perp+\cdots+g_1\hg^\perp+\hg^\perp \\
&\subseteq\hg^\perp,
\end{align*}
since we have noted above that $G\hg^\perp\subseteq\hg^\perp$. 
Consequently, if $\xi\in\gg^*$ and $\hg\subseteq\gg(\xi)$, then $G\xi\subseteq\xi+\hg^\perp$, which, by \eqref{quasi_eq1}, implies \eqref{quasi_eq2}. 
\end{remark}

\begin{lemma}\label{codim1_lemma2}
	Assume the setting of Lemma~\ref{codim1_lemma1} and let $\zg$ be the centre of~$\ng$. 
	If $\Oc\in(\gg^*/G)^\sim$ such that $\Oc\subseteq\gg^*$ is open, then the following assertions hold true: 
	\begin{enumerate}[{\rm(i)}]
		\item\label{codim1_lemma2_item1} We have $\sigma(D\vert_\zg)\cap\ie\RR=\emptyset$. 
		\item\label{codim1_lemma2_item2}
		We have $\Oc\cap\zg^\perp=\emptyset$. 
		\item\label{codim1_lemma2_item3}
		If $\xi\in\Oc$ 
		then 
		$G(\xi)=\{(x,0)\in G=N\rtimes_{\alpha}\RR\mid \xi\circ\ad_\ng x=0\text{ and }\langle\xi,Dx\rangle=0\}$ 
		and 
		$G(\xi\vert_\ng)=N(\xi\vert_\ng)\times\{0\}$.
		\item\label{codim1_lemma2_item4} 
		For every $\xi\in\Oc$ the subset $G \xi\vert_\ng =\{g\xi\vert_\ng\mid g\in G\}\subseteq \ng^*$ is open.  
	\end{enumerate}  
\end{lemma}

\begin{proof}
\eqref{codim1_lemma2_item1} 
We argue by contradiction. 
If $\sigma(D\vert_\zg)\cap\ie\RR\ne\emptyset$, then there exists a linear subspace $\{0\}\ne\zg_0\subseteq\zg$ with $D\zg_0\subseteq\zg_0$ such that either $D\vert_{\zg_0}=0$ (if $0\in \sigma(D\vert_\zg)\cap\ie\RR$), 
or $\dim\zg_0=2$ and $D\vert_{\zg_0}=\begin{pmatrix}
0 & -t_0 \\
t_0 &\hfill 0
\end{pmatrix}$ 
with respect to a suitable basis of $\zg_0$, 
where $t_0\in\RR\setminus\{0\}$ (if $\pm \ie t_0\in \sigma(D\vert_\zg)\cap\ie\RR\setminus\{0\}$). 
In any case, there exists a norm $\Vert\cdot\Vert$ on~$\zg_0$ 
such that the operator $\ee^{tD}\vert_{\zg_0}\colon\zg_0\to\zg_0$ is an isometry for every $t\in\RR$. 
We denote again by $\Vert\cdot\Vert$ the dual norm on~$\zg_0^*$. 

Now select any $\xi\in\Oc$, hence $G\xi\subseteq\Oc=r(\xi)\subseteq\overline{G\xi}$. 
We note that, since $D\zg_0\subseteq\zg_0$ and $[\ng,\zg_0]=\{0\}$, 
we have $[\gg,\zg_0]\subseteq\zg_0$. 
Therefore, if $\xi\in\zg_0^\perp$ then 
we may apply  \eqref{quasi_eq2} for $\hg=\zg_0$ to obtain 
$\Oc\subseteq \zg_0^\perp$, 
which is a contradiction with the fact that $\Oc\subseteq\gg^*$ is open while $\zg_0^\perp\subsetneqq\gg^*$ since $\dim\zg_0\ge 1$. 
Consequently $\xi\not\in\zg_0^\perp$, that is, $\xi\vert_{\zg_0}\ne0$. 

For arbitrary $(x,t)\in G$ we obtain by Lemma~\ref{codim1_lemma1} 
\begin{equation}\label{codim1_lemma2_proof_eq1}
\Ad_G^*((x,t)^{-1})\xi\vert_{\zg_0}
=\ee^{tD^*}\Ad_N^*(-x)\xi\vert_{\zg_0}
=\xi\circ \Ad_N(x)\circ \ee^{tD}\vert_{\zg_0}
=\xi\circ \ee^{tD}\vert_{\zg_0}
\end{equation}
where the last equality follows by the fact that $D(\zg_0)\subseteq\zg_0$ and $\Ad_N(x)\vert_\zg=\id_\zg$ since $x\in \ng$. 
Therefore, by the way $\zg_0$ was chosen, we obtain 
$\Vert \eta\vert_{\zg_0}\Vert=\Vert \xi\vert_{\zg_0}\Vert$ for every $\eta\in G\xi$, and this equality extends by continuity to every $\eta\in\overline{G\xi}$. 
In particular $\Vert \eta\vert_{\zg_0}\Vert=\Vert \xi\vert_{\zg_0}\Vert$ 
for every $\xi\in\Oc$. 
But this is a contradiction with the fact that $\Oc\subseteq\gg^*$ is an open subset, while the restriction mapping $\gg^*\to\zg_0^*$, $\eta\mapsto\eta\vert_{\zg_0}$, is an open mapping and every sphere in $\zg_0^*$ with respect to any norm has empty interior. 
This completes the proof of the fact that $\sigma(D\vert_\zg)\cap\ie\RR\ne\emptyset$.

\eqref{codim1_lemma2_item2} 
Since $[\gg,\zg]\subseteq\zg$ we obtain $G\zg^\perp\subseteq\zg^\perp$ 
hence, if $\Oc\cap\zg^\perp\ne\emptyset$ then $\Oc\subseteq\zg^\perp$, 
which is a contradiction since $\Oc$ is open and $\dim\zg\ge 1$ (just as in the proof of Lemma~\ref{codim1_lemma1}). 
Thus $\Oc\cap\zg^\perp=\emptyset$. 

\eqref{codim1_lemma2_item3}
Let $(x,t)\in G(\xi)$ be arbitrary. 
Then $(x,t)^{-1}\in G(\xi)$,  hence $\Ad_G^*((x,t)^{-1})\xi=\xi$. 
Restricting both sides of this equality to $\zg$ we obtain, just as in~\eqref{codim1_lemma2_proof_eq1} above, $\xi\vert_\zg=\xi\circ\ee^{tD}\vert_\zg$. 
Since $\sigma(D\vert_\zg)\cap\ie\RR\ne\emptyset$, it then follows that $t=0$. 
Now, writing $\xi=\begin{pmatrix}
\xi\vert_\ng \\
r_0
\end{pmatrix}\in\ng^*\dotplus\RR=\gg^*$, we obtain by Lemma~\ref{codim1_lemma1}, 
$$
\Ad_G^*((x,0)^{-1})\xi
=\begin{pmatrix}
\Ad_N^*(-x) & 0 \\
\psi(\ad_\ng x)Dx & 1
\end{pmatrix}
\begin{pmatrix}
\xi\vert_\ng \\
r_0
\end{pmatrix} \\
=\begin{pmatrix}
\Ad_N^*(-x)\xi\vert_\ng \\
\langle\xi,\psi(\ad_\ng x)Dx\rangle+r_0
\end{pmatrix}
$$
hence the equation $\Ad_G^*((x,0)^{-1})\xi=\xi$ is equivalent to $\xi\vert_\ng=\Ad_N^*(-x)\xi\vert_\ng $ and $\langle\xi,\psi(\ad_\ng x)Dx\rangle=0$. 
The Lie group  $N$ is nilpotent, therefore the equality 
$\xi\vert_\ng=\Ad_N^*(-x)\xi\vert_\ng $ is equivalent to 
$\xi\circ\ad_\ng x=0\in\ng^*$.
Then $\xi\circ\psi(\ad_\ng x)=\xi$, 
hence $(\xi\circ\psi(\ad_\ng x)\circ D)(x)=\langle\xi,Dx\rangle$. 
Thus the condition $(x,0)\in G(\xi)$ is equivalent to the pair of equations 
$\langle\xi,Dx\rangle=0$ and $\xi\circ\ad_\ng x=0$. 

Now let  $(x,t)\in G(\xi\vert_\ng)$ be arbitrary. 
Then  $(x,t)^{-1}\in G(\xi\vert_\ng)$, that is,  $$\Ad_G^*((x,t)^{-1})\xi\vert_\ng=\xi\vert_\ng.$$ 
This last equality is equivalent to 
$t=0$ and $\xi\vert_\ng=\Ad_N^*(-x)\xi\vert_\ng $ just as above. 

\eqref{codim1_lemma2_item4}
The restriction map $\rho\colon \gg\to\ng^*$, $\rho(\xi):=\xi\vert_\ng$, is $G$-equivariant since $\ng$ is an ideal of $\gg$. 
For arbitrary $\xi\in\Oc$, using the inclusions $G\xi\subseteq\Oc\subseteq\overline{G\xi}$ we then obtain 
\begin{equation}
\label{codim1_lemma2_proof_eq-3}
G\rho(\xi)=\rho(G\xi)\subseteq\rho(\Oc)\subseteq\rho(\overline{G\xi})
\subseteq\overline{\rho(G\xi)}=\overline{G\rho(\xi)}
\end{equation}
where the last inclusion follows from the fact that the mapping $\rho$ is continuous. 
Since the subset $\Oc\subseteq\gg^*$ is open and $\rho$ is an open mapping, 
it also follows from the above inclusions that $\rho(\xi)$ belongs to the interior of $\overline{G\rho(\xi)}$ for arbitrary $\xi\in\Oc$. 

We now claim that the subset $G\rho(\xi)=G\xi\vert_\ng\subseteq\ng^*$ is locally closed for 
every $\xi\in\Oc$. 
To this end we must show that the continuous bijective mapping 
\begin{equation}\label{codim1_lemma2_proof_eq-2}
G/G(\xi\vert_\ng)\to \ng^*,\quad gG(\xi\vert_\ng)\mapsto g\xi\vert_\ng
\end{equation}
is a homeomorphism,  
that is, if $g=(x,t)\in G$ and $g_k=(x_k,t_k)\in G$ for $k\ge 1$ is a sequence 
with $\lim\limits_{k\to\infty}g_k^{-1}\xi\vert_\ng=g^{-1}\xi\vert_\ng$ in $\ng^*$,  
then  $\lim\limits_{k\to\infty}g_k^{-1}G(\xi\vert_\ng)
=g^{-1}G(\xi\vert_\ng)$ in $G/G(\xi\vert_\ng)$. 
By Lemma~\ref{codim1_lemma1} we have 
\begin{equation}\label{codim1_lemma2_proof_eq-1}
\xi\circ\Ad_N(-x)\circ\ee^{tD}
=\lim\limits_{k\to\infty}\xi\circ\Ad_N(-x_k)\circ\ee^{t_kD}.
\end{equation}
Restricting the above equality to $\zg$ and using that $D(\zg)\subseteq\zg$ while $\Ad_N(-x)\vert_\zg=\Ad_N(-x_k)\vert_\zg=\id_\zg$ for all $k\ge 1$, we then obtain 
\begin{equation*}
\xi\circ\ee^{tD}\vert_\zg=\lim\limits_{k\to\infty}\xi\circ\ee^{t_kD}\vert_\zg.
\end{equation*}
Here $\xi\in\Oc\subseteq\gg^*\setminus\zg^\perp$ by \eqref{codim1_lemma2_item2},  while 
$\sigma(D\vert_\zg)\cap\ie\RR=\emptyset$ by Lemma~\ref{codim1_lemma2}, 
hence also $\sigma((D\vert_\zg)^*)\cap\ie\RR=\emptyset$. 
It follows by Lemma~\ref{loc-closed} that   $\lim\limits_{k\to\infty}t_k=t$. 
This further implies by \eqref{codim1_lemma2_proof_eq-1} that 
$\Ad_N^*(x)(\xi\vert_\ng)=\lim\limits_{k\to\infty}\Ad_N^*(x_k)(\xi\vert_\ng)$ in $\ng^*$. 
Since the coadjoint orbits of the nilpotent Lie group $N$ are closed in $\ng^*$, it follows that the mapping $N/N(\xi\vert_\ng)\to \ng^*$, 
$y\mapsto \Ad_N^*(y)(\xi\vert_\ng)$ is a homeomorphism onto its image, 
hence $xN(\xi\vert_\ng)=\lim\limits_{k\to\infty} x_k N(\xi\vert_\ng)$ 
in $N/N(\xi\vert_\ng)$. 

On the other hand, by \eqref{codim1_lemma2_item3}, the map  
\begin{equation*}
G/G(\xi\vert_\ng)\to  N(\xi\vert_\ng)\times\RR, \quad 
(y,s)G(\xi\vert_\ng)\mapsto (yN(\xi\vert_\ng),s)
\end{equation*}
is a homeomorphism, 
which shows that $(x,t)G(\xi\vert_\ng)
=\lim\limits_{k\to\infty}(x_k,t_k)G(\xi\vert_\ng)=\infty$ 
in the locally compact space $G/G(\xi\vert_\ng)$, 
and this completes the proof of the fact that the mapping \eqref{codim1_lemma2_proof_eq-2} is a homeomorphism onto its image.

We have just proved our claim that the $G$-orbit $G\rho(\xi)\subseteq\ng^*$ is locally closed for every 
$\xi\in\Oc$.   
Then, there exists an open subset $V\subseteq\ng^*$ with 
$G\rho(\xi)=V\cap\overline{G\rho(\xi)}$. 
However, by \eqref{codim1_lemma2_proof_eq-3}, 
the open subset $\rho(\Oc)\subseteq\ng^*$ satisfies $\rho(\xi)\in\rho(\Oc)\subseteq\overline{G\rho(\xi)}$, 
hence also $g\rho(\xi)\in g\rho(\Oc)\subseteq\overline{G\rho(\xi)}$ 
for all $g\in G$. 
Defining $W:=\bigcup\limits_{g\in G}g\rho(\Oc)$, 
it follows that $W\subseteq\ng^*$ is an open subset with $GW\subseteq W$ 
and $G\rho(\xi)\subseteq W\subseteq \overline{G\rho(\xi)}$. 
This implies $V\cap G\rho(\xi)\subseteq V\cap W\subseteq V\cap \overline{G\rho(\xi)}$ hence $G\rho(\xi)\subseteq V\cap W\subseteq  G\rho(\xi)$, and then $G\rho(\xi)\subseteq V\cap W$. 
It follows that the subset $G\rho(\xi)\subseteq\ng^*$ is open, 
which completes the proof. 
\end{proof}

\begin{lemma}\label{codim1_lemma3}
Let $G$ be a 1-connected solvable Lie group with its nilradical $N$ and let $Z$ be the centre of $N$. 
If $\dim(G/N)=1$, then the following conditions are equivalent: 
\begin{enumerate}[{\rm(i)}]
	\item\label{codim1_lemma3_item1} The Lie group $G$ has open coadjoint quasi-orbits. 
	\item\label{codim1_lemma3_item2} The Lie group $G$ has open coadjoint orbits. 
	\item\label{codim1_lemma3_item3} The centre of $G$ is trivial, the generic coadjoint orbits of $N$ are flat, and $\dim Z=1$.
\end{enumerate}
If these conditions are satisfied, then every open coadjoint quasi-orbit is a coadjoint orbit.
\end{lemma}

\begin{proof}
We have \eqref{codim1_lemma3_item2}$\iff$\eqref{codim1_lemma3_item3} 
by \cite[Lemma 3.7]{BB23}, while the implication  
\eqref{codim1_lemma3_item2}$\implies$\eqref{codim1_lemma3_item1} 
is trivial. 

``\eqref{codim1_lemma3_item1}$\implies$\eqref{codim1_lemma3_item3}'' 
Let $\Oc$ be an arbitrary  coadjoint quasi-orbit which is an open subset of $\gg^*$, 
and let $\xi\in\Oc$.
Then the action $G\times\ng^*\to\ng^*$, $(g,\eta)\mapsto g\eta=\eta\circ\Ad_G(g^{-1})\vert_\ng$, is transitive 
and gives the diffeomorphism $G/G(\xi\vert_\ng)\to G \xi\vert_\ng\subset \ng^*$, $gG(\xi\vert_\ng)\mapsto g\xi\vert_\ng$. 
Since the set $G \xi\vert_\ng$ is open in $\ng^*$,  by Lemma~\ref{codim1_lemma2}\eqref{codim1_lemma2_item4},  it has the same dimension as $\ng^*$.
Therefore $\dim\ng^*=\dim G-\dim G(\xi\vert\ng)=1+\dim\ng-\dim G(\xi\vert\ng)$, hence $\dim G(\xi\vert\ng)=1$. 
By  Lemma~\ref{codim1_lemma2}\eqref{codim1_lemma2_item3} we then obtain 
$\dim N(\xi\vert_\ng)=1$, which directly implies that the nilpotent Lie group $N$ has 1-dimensional centre and generic coadjoint orbits. 

Finally, if \eqref{codim1_lemma3_item1}-- \eqref{codim1_lemma3_item3} hold, 
then every  open coadjoint quasi-orbit is a coadjoint orbit by Corollary~\ref{minimal_cor}. 
\end{proof}

\begin{proof}[Proof of Theorem~\ref{codim1_th}] 
	By Remark~\ref{isolated}, we must prove that every isolated point of $\Prim(G)$ is type~$\I$. 
	
	Let $\Pc\in\Prim(G)$ be an isolated point.
By Remark~\ref{isolated}, there exists a unique square-integrable class
$$[\pi]^\frown\in \stackrel{\frown}{G}_\nor
\text{ with }
\Pc=\ker([\pi]^\frown):=\Ker\pi\subseteq C^*(G).$$
By Theorem~\ref{open_th} we obtain a coadjoint quasi-orbit $\Oc\in(\gg^*/G)^\sim$ such that $\Oc\subseteq\gg^*$ is open and $[\pi]^\frown=\ell(\O)$, where $\O:=\Bun(\Oc)\in(\Bun(\Oc)/G)^\approx$.

Since $\dim(G/N)=1$ by hypothesis, we have $\Oc\in\gg^*/G$ by Lemma~\ref{codim1_lemma3}, and then $\ell(\O)$ is type~$\I$ by Theorem~\ref{F_gen}.
Therefore the ideal $\Pc$ is type~$\I$.
\end{proof}

\section{Examples}
\label{sect7}

In this section we give examples of groups to which our results apply. 
In particular, we construct a family of solvable Lie groups $G$ which have open coadjoint quasi-orbits that are not orbits, and have square-integrable representations that are not type~$\I$. 
These examples are not unimodular groups, as established in general in Corollary~\ref{clopen_cor}. 
They show in particular that the codimension of the nilradical~$N$ can be any integer at least~$3$. 
We point out that necessarily $\dim(G/N)\ge2$ by Theorem~\ref{codim1_th}, but we do not know if such examples exist with $\dim(G/N)=2$.

\subsection*{Semidirect products of abelian Lie groups}
\normalfont
Let $\Vc$ be a finite-dimensional real vector space and let $\ag$ be a Lie algebra with its corresponding 1-connected Lie group~$A$. 
If $\alpha\colon A\to \GL(\Vc)$ is a continuous group homomorphism, 
we form the semidirect product 
$G:=\Vc\rtimes_\alpha A$ with the group operation $(v,a)\cdot (w,b)=(v+\alpha(a)w, ab)$ for all $v,w\in\Vc$ and $a,b\in A$. 
For every $p\in \Vc^*$ we define $\theta_p\colon \Vc\to\ag^*$, $\theta_p(v):=-\langle p,\de\pi(\cdot)v\rangle$, where $\langle\cdot,\cdot\rangle\colon\Vc^*\times\Vc\to\RR$ is the duality pairing. 
Then $\gg^*=\Vc^*\times\ag^*$, and the coadjoint action $\Ad_G^*\colon G\times(\Vc^*\times\ag^*)\to \Vc^*\times\ag^*$ is given by the formula 
\begin{equation*}
(\Ad_G^*(v,a))(p,\xi)=(\alpha(a^{-1})^*p,\Ad_A^*(a)\xi-\theta_{\alpha(a^{-1})^*p}(v))
\end{equation*}
for all $v\in\Vc$, $a\in A$, $p\in\Vc^*$, and $\xi\in\ag^*$. 
(See \cite[Rem.~4.15]{BB21a}  and the references therein.)

For arbitrary $x\in \ag$ we have 
\begin{align*}
\langle \theta_{\alpha(a^{-1})^*p}(v),x\rangle 
&=-\langle \alpha(a^{-1})^*p,\de\alpha(x)v\rangle \\
&=-\langle p,\alpha(a^{-1})\de\alpha(x)\alpha(a)\alpha(a^{-1})v\rangle \\
&=-\langle p,\de\alpha(\Ad_A(a^{-1})x)\alpha(a^{-1})v\rangle \\
&=\langle \theta_p(\alpha(a^{-1})v),\Ad_A(a^{-1})x\rangle 
\end{align*}
where we have denoted by $\langle\cdot,\cdot\rangle\colon\ag^*\times\ag\to\RR$ the duality pairing for $\ag$ as well. 
Consequently 
\begin{equation*}
\theta_{\alpha(a^{-1})^*p}(v)=\theta_p(\alpha(a^{-1})v)\circ\Ad_A(a^{-1}). 
\end{equation*}
If $A$ is abelian, then $\Ad_A(a^{-1})=\id_\ag$ and $\Ad_A^*(a)=\id_{\ag^*}$, hence we obtain 
\begin{equation*}
(\Ad_G^*(v,a))(p,\xi)
=(\alpha(a^{-1})^*p,\xi-\theta_{\alpha(a^{-1})^*p}(v))
=(\alpha(a^{-1})^*p,\xi-\theta_p(\alpha(a^{-1})v))
\end{equation*}
and therefore 
\begin{align*}
G(p,\xi)&=(\Ker\theta_p)\rtimes_\alpha A(p),\\
\Ad_G^*(G)(p,\xi) & =\alpha^*(A)p\times(\xi+\theta_p(\Vc)),
\end{align*}
where $\alpha^*\colon A\to \GL(\Vc^*)$, $\alpha^*(a):=\alpha(a^{-1})^*$, 
and $A(p):=\{a\in A\mid \alpha^*(a)p=p\}$. 

We now note the following consequences of the above remarks, in the case where $A$ is abelian:
\begin{enumerate}[{\rm(i)}]
	\item 
	The coadjoint orbit $\Ad_G^*(G)(p,\xi)\subseteq\gg^*$ is locally closed if and only if the orbit $\alpha^*(A)p\subseteq\Vc^*$ is locally closed. 
	\item We have $G(p,\xi)_\1=(\Ker\theta_p)\rtimes_\alpha A(p)_\1$. 
	\item We have 
	$$\overline{\Ad_G^*(G)(p,\xi)}=\overline{\Ad_G^*(G)(q,\eta)}
	\iff\begin{cases}
	\overline{\alpha^*(A)p}=\overline{\alpha^*(A)q}, \\
	\xi-\eta\in\theta_p(\Vc). 
	\end{cases}$$
	\item The coadjoint quasi-orbit of $(p,\xi)$ in $(\gg^*/G)^\sim$ is open in $\gg^*$ if and only if the quasi-orbit of $p$ in $(\Vc^*/A)^\sim$ is open in $\Vc^*$ and $\theta_p(\Vc)=\ag^*$. 
	\item If $A(p)=\{\1\}$, then $G(p^{-1})$ is connected, hence $\tau\vert_{\Bun(\Oc)}\colon\Bun(\Oc)\to\Oc$ is bijective, where $\Oc$ is the coadjoint quasi-orbit of $(p,\xi)$ in $(\gg^*/G)^\sim$. 
	\end{enumerate}

Let us now consider some specific cases of semidirect products of abelian Lie groups.

\begin{example}
\label{ax+b_complex}
\normalfont
Let $\gg$ be the real Lie algebra defined by a basis $X_1,X_2,X_3,X_4$ satisfying the commutation relations 
$$[X_3,X_1]=X_1,\ [X_3,X_2]=X_2,\ [X_4,X_1]=X_2,\ [X_4,X_2]=-X_1.$$
If we define the complex Lie algebra structure on $\CC^2$ by 
\begin{equation}
\label{ax+b_complex_eq1}
[(w,z),(w',z')]:=(wz'-zw',0)\text{ for all }(w,z),(w',z')\in\CC^2
\end{equation}
then the $\RR$-linear isomorphism 
$$\Psi\colon \gg\to\CC^2,\quad b_1X_1+b_1X_2+b_3X_3+b_4X_4\mapsto (b_1+\ie b_2,b_3+\ie b_4)$$
is an isomorphism of real Lie algebras. 
Therefore we will write $\gg=\CC^2$ with the above Lie bracket~\eqref{ax+b_complex_eq1} (regarded however as a real Lie algebra). 
We also have 
$$\dg=[\gg,\gg]=\CC\times\{0\}.$$
A solvable Lie group $G$ whose Lie algebra is isomorphic to $\gg$ is  
$G:=(\CC^2,\cdot)$ with its group operations given by 
$$(b,a)\cdot (b',a'):=(b+\ee^{a}b',a+a')\text{ and }(b,a)^{-1}:=(-\ee^{-a}b,-a)$$ 
for all $a,a',b,b'\in\CC$.
The centre of the group $G$ is $Z_G=\{0\}\times2\pi\ie\ZZ$.

The coadjoint action of $G$ is given by 
$$\Ad_G((b,a)^{-1})\colon \gg^*\to\gg^*,\quad (\Ad^*_G((b,a)^{-1}))(w',z')=(\overline{\ee^{a}}w',-\overline{b}w'+z'), $$
where we have identified $\gg$ to its real dual space $\gg^*$ via the duality pairing 
$$\gg\times\gg\to\RR, \quad \langle (w,z), (w',z')\rangle:=\Re(\overline{w}w'+\overline{z}z').$$
This shows that the set $\dg^\perp=\{0\}\times\CC$ is the set of fixed points of the coadjoint action of $G$, 
while the complement of this set, to be denoted 
$$\Oc_0:=\CC^\times\times\CC$$
is an open coadjoint orbit of $G$. 
We have  
$$\Bun(\gg^*)=\Bun(\Oc_0)\sqcup\dg^\perp
\text{ and }\RelS=\Bun(\Oc_0)/G \sqcup\dg^\perp.$$ 
Here $\Bun(\Oc_0)/G$ is a compact and open set, homeomorphic with a torus $\TT$. 
Therefore, Theorem~\ref{homeo}  shows that
 $$\Prim(G) =  S_1  \cup S_2 $$
 where $S_1$ is a dense, open and compact subset of $\Prim(G)$, homeomorphic with $\TT$, while 
 $S_2= \{0\}\times \CC$. 
\end{example}

In Example~\ref{codim3} below, we consider a 7-dimensional Lie group that was briefly mentioned in \cite[Subsect. 3.11]{Ros78}, 
and we will see in Example~\ref{codim3bis} that this group is only the first in an infinite sequence of solvable Lie groups that have square-integrable representations which are not type~$\I$.

\begin{example}\label{codim3}
\normalfont
Let $A=(\RR^3,+)$, $N=(\CC^2,+)$, $\theta\in\RR\setminus\QQ$, and 
$$\alpha\colon A\to\Aut(N), \quad 
\alpha(r,s,t)=
\begin{pmatrix}
\ee^{r+\ie t} & 0 \\0 & \ee^{s+\ie \theta t}
\end{pmatrix}.$$
Then the semidirect product $G:=N\rtimes_\alpha A$ is a solvable Lie group with trivial centre, with its nilradical $N$ satisfying $\dim(G/N)=3$, 
and we show in Example~\ref{codim3bis} in a more general setting that there exists a coadjoint quasi-orbit $\Oc\in(\gg^*/G)^\sim$ with the following properties: 
\begin{itemize}
	\item $\Oc\subseteq\gg^*$ is a dense open subset; 
	\item $\Oc\not\in\gg^*/G$;  
	\item $G(\xi)=\{\1\}$ for every $\xi\in\Oc$. 
\end{itemize}
Then the condition $(\Bun(\Oc)/G)^\approx=\{\Bun(\Oc)\}$ is trivially satisfied. 
In particular, by Theorem~\ref{open_th}, the quasi-equivalence class $\ell(\O)\in \stackrel{\frown}{G}_\nor$ is square integrable for $\O=\Bun(\Oc)$. 
On the other hand, by Theorem~\ref{F_gen}, $\ell(\O)$ is not type~$\I$ since $\Oc\not\in\gg^*/G$. 
\end{example}

\begin{example}[Generalization of Example~\ref{codim3}]
	\label{codim3bis}
	\normalfont 
We consider the abelian Lie group $A:=(\RR^k,+)$ and the real vector space $\Vc:=\CC^n$. 
Select $\xi_1,\dots,\xi_n,\eta_1,\dots,\eta_n\in\RR^k$ satisfying the following conditions: 
\begin{enumerate}[{\rm(a)}]
	\item\label{item_a} We have $\spa(\{\xi_j\mid j=1,\dots,n\}\cup\{\eta_j\mid j=1,\dots,n\})=\RR^k$, hence $2n\ge k$.
	\item\label{item_b} The vectors $\eta_1,\dots,\eta_n\in\RR^k$ are linearly independent. 
	\item\label{item_c} The mapping 
	$\Psi\colon \{\eta_1,\dots,\eta_n\}^\perp\to\TT^n$, $\Psi(a):=(\ee^{\ie\langle\xi_1,a\rangle},\dots,\ee^{\ie\langle\xi_n,a\rangle})$, 
	is injective, and its image is dense in, and different from, the torus $\TT^n$. 
\end{enumerate}
We note that \eqref{item_c} implies $n\ge 2$ and $k-n\mathop{=}\limits^{\eqref{item_b}}\dim\{\eta_1,\dots,\eta_n\}^\perp\ge 1$,  
hence $k\ge 3$. 
(We will also see below that these conditions imply $2n>k$.) 
We now define 
$\alpha\colon \RR^k\to\GL(n,\CC)\subseteq\GL(\Vc)$ by 
$$\alpha(a)=\begin{pmatrix}
\ee^{\langle\eta_1,a\rangle+\ie\langle\xi_1,a\rangle} & & 0 \\
 & \ddots & \\
0 & &\ee^{\langle\eta_n,a\rangle+\ie\langle\xi_n,a\rangle}
\end{pmatrix}$$
where we have denoted by $\langle\cdot,\cdot\rangle\colon\RR^k\times\RR^k\to\RR$ the canonical real scalar product. 
If we form the corresponding semidirect product $G:=\Vc\rtimes_\alpha A$, 
then we claim that $G$ is a solvable Lie group which has a  coadjoint quasi-orbit $\Oc\in(\gg^*/G)^\sim$ with the following properties: 
\begin{enumerate}[{\rm(i)}]
	\item\label{item_i} $\Oc\subseteq\gg^*$ is a dense open subset; 
	\item\label{item_ii} $\Oc\not\in\gg^*/G$; 
	\item\label{item_iii} $G(p,\xi)$ is abelian and connected for every $(p,\xi)\in\Oc\subseteq\Vc^*\times\ag^*=\gg^*$. 
\end{enumerate}
In particular, the mapping $\tau\vert_{\Bun(\Oc)}\colon\Bun(\Oc)\to\Oc$ is bijective and, for $\O:=\Bun(\Oc)$,  we have $\O\in(\Bun(\Oc)/G)^\approx$, 
and its corresponding quasi-equivalence class $\ell(\O)\in \stackrel{\frown}{G}_\nor$ is square integrable by Theorem~\ref{open_th} and not type~$\I$ by Theorem~\ref{F_gen}. 

To prove the existence of $\Oc\in(\gg^*/G)^\sim$ with the aforementioned properties, 
we first note the equality 
$$\Ker\de\alpha=\{\xi_j\mid j=1,\dots,k\}^\perp\cap\{\eta_j\mid j=1,\dots,k\}^\perp$$
hence condition~\eqref{item_a} is equivalent to the fact that the centre of $G$ is trivial. 

We now denote by $(\cdot\mid\cdot)\colon\CC^n\times\CC^n\to\CC$ the canonical complex scalar product, antilinear in its first variable, and let $\langle\cdot,\cdot\rangle:=\Re(\cdot\mid\cdot)$, which is a real scalar product on~$\CC^n$. 

We use the above real scalar products in order to perform the identifications $\Vc^*=\Vc$ and $\ag^*=\ag$. 
Then for every 
$p\in\Vc^*=\CC^n$ the operator $\theta_p\colon \Vc=\CC^n\to\ag^*=\RR^k$ satisfies for all $a\in\ag=\RR^k$
\allowdisplaybreaks
\begin{align*}
-\langle\theta_p(v),a\rangle 
&=\langle p,\de\alpha(a)v\rangle 
=\Re (p\mid\de\alpha(a)v)  \\
&=\sum_{r=1}^n\Re(\overline{p_r}(\langle\eta_r,a\rangle+\ie\langle\xi_r,a\rangle) v_r) \\
&=\sum_{r=1}^n\langle\eta_r,a\rangle\Re(\overline{p_r} v_r)
-\langle\xi_r,a\rangle\Im(\overline{p_r} v_r)
\end{align*}
where we have written $p=(p_1,\dots,p_n),v=(v_1,\dots,v_n)\in\CC^n$. 
Thus 
\begin{equation}\label{thetap}
\theta_p(v)=\sum_{r=1}^n \Im(\overline{p_r} v_r)\xi_r-\Re(\overline{p_r} v_r)\eta_r.
\end{equation}
If moreover 
$p\in(\CC^\times)^n$, then 
\begin{align*}
A(p)& =\{a\in\RR^k\mid\alpha(a)=\1\} \\
&=\{a\in\RR^k\mid\langle\xi_j,a\rangle\in2\pi\ie\ZZ\text{ and }\langle\eta_j,a\rangle=0\text{ for }j=1,\dots,k\} \\
&=\{0\}
\end{align*}
where the last equality follows by hypothesis~\eqref{item_c}. 
We now fix $p=(p_1,\dots,p_n)\in(\CC^\times)^n$ and $\xi\in\RR^k$ 
and let $\Oc\in(\gg^*/G)^\sim$ be the coadjoint quasi-orbit of $(p,\xi)\in\CC^n\times\RR^k=\gg^*$. 
We proceed to show that $\Oc=(\CC^\times)^n\times\RR^k$ and $\Oc$ has properties \eqref{item_i}--\eqref{item_iii}. 

In fact, \eqref{item_iii} holds since 
$G(p,\xi)=(\Ker\theta_p)\rtimes A(p)=(\Ker\theta_p)\times\{0\}$ 
is isomorphic to $\Ker\theta_p$, which is a linear subspace of $\Vc$, hence  connected. 
 Let us denote $E:=\spa\{\eta_1,\dots,\eta_n\}\subseteq\RR^k$.
Hypothesis~\eqref{item_b} implies that $\alpha\vert_E$ is an isomorphism of $(E,+)$ onto the group of all diagonal matrices in $\GL(n,\CC)$ whose diagonal entries are strictly positive. 
On the other hand, hypothesis~\eqref{item_c} implies that $\alpha(E^\perp)p$ is dense in, and different from, the torus $(\vert p_1\vert\TT)\times\cdots\times(\vert p_n\vert\TT)$. 
It is then straightforward to check that $\alpha(A)p$ is a dense open subset of $\CC^n$ which is not locally closed. 
On the other hand, it follows by hypothesis~\eqref{item_a} along with $p\in(\CC^\times)^n$ and \eqref{thetap} that $\theta_p(\Vc)=\ag^*$, 
hence $\Oc=(\CC^\times)^n\times\RR^k$ and  \eqref{item_i} holds true. 
Finally, we prove~\eqref{item_ii} by contradiction. 
Assuming $\Oc\in\gg^*/G$, it follows that $\Oc$ is the coadjoint orbit of $(p,\xi)\in\Oc$. 
Then, since $\Oc\subseteq\gg^*$ is open, we obtain $G(p,\xi)=\{\1\}$, 
hence $\Ker\theta_p=\{0\}$. 
Since we have seen above that $\theta_p\colon \CC^n\to\RR^k$ is surjective, we then obtain $2n= k$. 
Then, by hypothesis~\eqref{item_a}, the vectors $\xi_1,\dots,\xi_n$ must be linearly independent, which implies that the mapping $\Psi$ from hypothesis~\eqref{item_c} is surjective, and this is a contradiction. 
Consequently \eqref{item_ii} holds true, and we are done. 
\end{example}

\subsection*{Acknowledgment}
We are grateful to Karl-Hermann Neeb for useful remarks on 1-connected Lie groups, 
and to Jordy van Velthoven for pointing out an inaccuracy in an earlier version of our paper, concerning the centre of the group~$G$ in Example~\ref{ax+b_complex}. 
We also thank the Referee for suggestions that improved our paper at many places.


\begin{thebibliography}{1000000}



	
\bibitem[ArCu20]{ArCu20}
\textsc{D.~Arnal, B.~Currey III}, 
\textit{Representations of solvable Lie groups}.  
New Mathematical Monographs \textbf{39}. Cambridge University Press, Cambridge, 2020.

 
\bibitem[AuKo71]{AuKo71}
\textsc{L.~Auslander, B.~Kostant}, 
Polarization and unitary representations of solvable Lie groups. 
\textit{Invent. Math.} \textbf{14} (1971), 255--354.

\bibitem[Bk04]{Bk04}
\textsc{B.~Bekka}, 
Square integrable representations, von Neumann algebras and an application to Gabor analysis. 
\textit{J. Fourier Anal. Appl.} \textbf{10} (2004), no. 4, 325--349. 


\bibitem[BB16]{BB16}
	\textsc{I.~Belti\c t\u a, D.~Belti\c t\u a}, 
On $C^*$-algebras of exponential solvable Lie groups and their real ranks. 
\textit{J. Math. Anal. Appl.} \textbf{437} (2016), no. 1, 51--58.


\bibitem[BB18a]{BB18a}
\textsc{I.~Belti\c t\u a, D.~Belti\c t\u a},  
$C^*$-dynamical systems of solvable Lie groups. 
\textit{Transform. Groups} \textbf{23} (2018), no. 3, 589--629. 

\bibitem[BB18b]{BB18b}
\textsc{I.~Belti\c t\u a, D.~Belti\c t\u a},  
Quasidiagonality of $C^*$-algebras of solvable Lie groups. 
\textit{Integral Equations Operator Theory} \textbf{90} (2018), no. 1, Paper No. 5, 21 pp. 

\bibitem[BB21a]{BB21a}
\textsc{I.~Belti\c t\u a, D.~Belti\c t\u a}, 
Linear dynamical systems of nilpotent Lie groups. 
\textit{J. Fourier Anal. Appl.} \textbf{27} (2021), no. 5, Paper No. 74, 29 pp. 




\bibitem[BB21b]{BB21b}
\textsc{I.~Belti\c t\u a, D.~Belti\c t\u a},  
AF-embeddability for Lie groups with $T_1$ primitive ideal spaces. 
\textit{J. Lond. Math. Soc. (2)} 
\textbf{104} (2021), no. 1, 320--340.

\bibitem[BB23]{BB23}
	\textsc{I.~Belti\c t\u a, D.~Belti\c t\u a}, 
	On stably finiteness for $C^*$-algebras of exponential solvable Lie groups. 
	\textit{Math. Z.} \textbf{304} (2023), no.~1, Paper No.~2.

\bibitem[Bo82]{B82}
\textsc{J.~Boidol}, 
Connected groups with polynomially induced dual. 
\textit{J. Reine Angew. Math.} \textbf{331} (1982), 32--46.

\bibitem[Bo06]{Bo06}
{\sc N.~Bourbaki}, 
\textit{El\'ements de Math\'ematique. Groupes et alg\`ebres de Lie}. Ch. 2--3, Springer, 2006. 

\bibitem[Bo07]{Bo07b}
\textsc{N.~Bourbaki}, 
\textit{El\'ements de Math\'ematique. Topologie g\'en\'erale}. Ch. 1--4, Springer, 2007. 


	
	
	\bibitem[CFT16]{CFT16}
\textsc{B.~Currey, H. F\"uhr, K.~Taylor},
 Integrable wavelet transforms with abelian dilation groups. 
 \textit{J. Lie Theory} \textbf{26} (2016), no. 2, 567--596.
	
	
\bibitem[CuOu12]{CuOu12} 
\textsc{B.~Currey, V.~Oussa}, 
Admissibility for monomial representations of exponential Lie groups. \textit{J. Lie Theory} \textbf{22} (2012), no. 2, 481--487.

\bibitem[Di77]{Di64}
\textsc{J.~Dixmier}, 
\textit{$C^*$-algebras.} 
North-Holland Mathematical Library \textbf{15}, 
North-Holland Publishing Co., Amsterdam-New York-Oxford, 1977. 

\bibitem[DuRa76]{DuRa76} 
\textsc{M. Duflo, M. Ra\"\i s}, 
Sur l'analyse harmonique sur les groupes de Lie r\'esolubles. 
\textit{Ann. Sci. \'Ecole Norm. Sup. (4)} \textbf{9} (1976), no. 1, 107--144.


\bibitem[Ei40]{Ei40}
 \textsc{S.~Eilenberg}, 
 On a theorem of P.A.~Smith concerning fixed points for periodic transformations. 
 \textit{Duke Math. J.}  \textbf{6} (1940), 428--437.
 
 \bibitem[Fu02]{Fu02} 
 \textsc{H. F\"uhr}, Admissible vectors for the regular representation. \textit{Proc. Amer. Math. Soc.} \textbf{130} (2002), no. 10, 2959--2970.
 
 \bibitem[Fu05]{Fu05} 
 \textsc{H. F\"uhr}, 
 \textit{Abstract harmonic analysis of continuous wavelet transforms}. 
 Lecture Notes in Mathematics, 1863. Springer-Verlag, Berlin, 2005.
 
 \bibitem[Fu15]{Fu15} 
 \textsc{H. F\"uhr},
 Coorbit spaces and wavelet coefficient decay over general dilation groups. \textit{Trans. Amer. Math. Soc.} \textbf{367} (2015), no. 10, 7373--7401.


\bibitem[FuO22]{FO22} 
\textsc{H. F\"uhr, V.~Oussa}, Groups with frames of translates. 
\textit{Colloq. Math.} \textbf{167} (2022), no.~1, 73--91. 

 
 \bibitem[FuV21]{FvV21} 
 \textsc{H. F\"uhr, J.T.~van~Velthoven}, 
 Coorbit spaces associated to integrably admissible dilation groups. 
 \textit{J. Anal. Math.} \textbf{144} (2021), no.~1, 351--395.


\bibitem[Gr80]{Gr80}
\textsc{Ph.~Green}, 
Square-integrable representations and the dual topology. 
\textit{J. Functional Analysis} \textbf{35} (1980), no. 3, 279--294.

\bibitem[HN12]{HN12}
	\textsc{J.~Hilgert, K.-H.~Neeb}, 
	\textit{Structure and geometry of Lie groups}. 
	Springer Monographs in Mathematics. Springer, New York, 2012.



\bibitem[KT96]{KT96}
\textsc{E.~Kaniuth, K.F.~Taylor}, 
Minimal projections in $L^1$-algebras and open points in the dual spaces of semi-direct product groups. 
\textit{J. London Math. Soc. (2)} \textbf{53} (1996), no. 1, 141--157. 

\bibitem[KT13]{KT13} 
\textsc{E.~Kaniuth, K.F.~Taylor}, 
\textit{Induced representations of locally compact groups}. 
Cambridge Tracts in Mathematics, 197. Cambridge University Press, Cambridge, 2013. 

\bibitem[Mo77]{Mo77}
\textsc{C.C.~Moore}, 
Square integrable primary representations. 
\textit{Pacific J. Math.} \textbf{70} (1977), no. 2, 413--427. 

\bibitem[Ne00]{Ne00}
\textsc{K.-H.~Neeb}, 
\textit{Holomorphy and convexity in Lie theory}. 
De Gruyter Expositions in Mathematics, 28. Walter de Gruyter \& Co., Berlin, 2000.

\bibitem[Pu71]{Pu71}
\textsc{L.~Puk\'anszky}, 
Unitary representations of solvable Lie groups. 
\textit{Ann. Sci. \'Ecole Norm. Sup. (4)} \textbf{4} (1971), 457--608.

\bibitem[Pu73]{Pu73}
{\sc L.~Puk\'anszky}, 
The primitive ideal space of solvable Lie groups. 
{\it Invent. Math.} {\bf 22} (1973), 75--118. 

\bibitem[Pu74]{Pu74}
{\sc L.~Puk\'anszky},  
Characters of connected Lie groups. 
{\it Acta Math.} {\bf 133} (1974), 81--137.

\bibitem[Pu86]{Pu86}
\textsc{L.~Puk\'anszky}, 
Quantization and Hamiltonian $G$-foliations.
\textit{Trans. Amer. Math. Soc.} \textbf{295} (1986), no. 2, 811--847.

\bibitem[Pu92]{Pu92}
\textsc{L.~Puk\'anszky},
On the coadjoint orbits of connected Lie groups.
\textit{Acta Sci. Math. (Szeged)} 
\textbf{56} (1992), no. 3--4, 347--358 (1993)

\bibitem[Pu99]{Pu99} 
\textsc{L.~Puk\'anszky}, 
\textit{Characters of connected Lie groups}. 
Mathematical Surveys and Monographs, 71. American Mathematical Society, Providence, RI, 1999. 


\bibitem[Ro78]{Ros78}
\textsc{J.~Rosenberg},
  Square-integrable factor representations of locally compact groups. 
  \textit{Trans. Amer. Math. Soc.} 
  \textbf{237} (1978), 1--33.
  











	
\end{thebibliography}
\end{document}